\documentclass{article}

\usepackage{blindtext}
\usepackage{subfiles} 
\usepackage{tikz}

\usepackage{packages/commonpackage}
\usepackage{packages/package_for_list_of_content}
\usepackage{packages/package_to_delete}
\usepackage{packages/macro}

\numberwithin{equation}{section}
\theoremstyle{plain}
\newtheorem{theorem}{Theorem}[section]

\newtheorem{proposition}[theorem]{Proposition}
\newtheorem{corollary}[theorem]{Corollary}
\newtheorem{lemma}[theorem]{Lemma}
\newtheorem{definition}[theorem]{Definition}
\newtheorem{remark}[theorem]{Remark}

\newtheorem{example*}{Example}
\newtheorem{claim}[theorem]{Claim}

\title{The (local) geometry of oscillatory integrals on manifolds: Dimension three}
\date{}
\author{Song Dai, Liuwei Gong and Shaoming Guo}

\begin{document}
\maketitle

\begin{abstract}
    Sogge \cite{Sog99} studied Kakeya problems on two extreme types of three dimensional Riemannian manifolds: Manifolds with the most symmetries (manifolds of constant sectional curvature) and manifolds with the least symmetries, which he called manifolds with chaotic curvature and variably curved manifolds. In the same paper, Sogge proposed studying manifolds with intermediate symmetry, such as (locally) symmetric spaces. 

    In the current paper, we propose a classification of curvature conditions in the spirit of Sogge's program. In particular, these curvature conditions give a complete geometric characterization of the contact order conditions (for Riemannian distance functions), introduced by the authors of \cite{DGGZ24} and \cite{CGG25} when studying H\"ormander-type oscillatory integral operators. (See Theorem \ref{260509theorem2_1}.)

    One of these conditions generalizes Sogge's chaotic curvature condition to all finite orders: The chaotic curvature condition of order $\le k$ for every $k\ge 1,$ with the case $k=1$ corresponding to Sogge's original condition for variably curved manifolds.  As byproducts of our main results (see Theorem \ref{existence}), we show that there are no manifolds satisfying the chaotic curvature condition of order $\le 1$. We also show that both the chaotic curvature condition of order $\le 2$ and its failure can occur robustly under small smooth perturbations, and for every $k\ge 3$, a ``generic" manifold satisfies the chaotic curvature condition of order $\le k$.

    It turns out that the chaotic curvature condition of order $\le k$ is precisely the same as the notion of non-$(k+2)$-exceptional,  where $k$-exceptional is introduced by Lytchak and Petrunin \cite{LP22} when studying convex sets and the non-existence of totally geodesic sub-manifolds.  Thus our results imply, in particular, that every manifold is $3$-exceptional. 
\end{abstract}

\section{Introduction}

We list some notation and make some conventions that will be used throughout the paper. 

\begin{enumerate}
    \item For $\epsilon>0$ and $\mathbf{x} \in \mathbb{R}^n$, we let $\mathbb{B}_\epsilon^n(\mathbf{x})$ denote the ball of radius $\epsilon$ in $\mathbb{R}^n$ centered at $\mathbf{x}$. If $\epsilon=1$, we often abbreviate $\mathbb{B}_1^n(\mathbf{x})$ to $\mathbb{B}^n(\mathbf{x})$; if $\mathbf{x}=0$, then we often abbreviate $\mathbb{B}_\epsilon^n(\mathbf{x})$ to $\mathbb{B}_\epsilon^n$.

    \item To simplify our discussions, we always assume that phase functions $\phi(x, t; \xi)$ are analytic and $x\in \R^2, t\in \R, \xi\in \R^2$. 

    \item For a given phase function $\phi(x, t; \xi)$, we always work with $(x, t)\in \B^{(3)}_{\epsilon_{\phi}}$ and $\xi\in \B^{(2)}_{\epsilon_{\phi}}$, where $\epsilon_{\phi}>0$ is a sufficiently small constant depending only on $\phi$. 

    \item Manifolds that appear in the current paper are always assumed to be analytic and three dimensional, unless otherwise stated. Analyticity is used to get rid of functions whose Taylor coefficients are all zero but the function itself is not identically zero. 

    \item  We use $f\equiv 0$ to mean that the function $f$ vanishes everywhere, and $f\equiv g$ to mean that these two functions are equal everywhere. 

    \item We largely follow the notation in the paper \cite{DGGZ24}. Moreover, for readers not used to curvature tensor calculations, it may be helpful to take a look at Subsection 3.1 and Subsection 5.1 in \cite{DGGZ24} for some preliminary calculations (we will also recall them whenever they are applied). 

    \item We usually use $\mathcal{R}$ for the Riemannian curvature tensor, unless otherwise defined. 
\end{enumerate}

\subsection{H\"ormander-type oscillatory integrals}

Let $\phi(x, t; \xi)$ be a phase function satisfying the standard H\"ormander non-degeneracy condition (see for instance (H1) and (H2) in Guth, Hickman and Iliopoulou \cite{GHI19}). Here $x\in \R^2, t\in \R, \xi\in \R^2$. We are interested in studying H\"ormander-type oscillatory integral operators 
\begin{equation}
T_N^{(\phi)} f(x, t):=
\int_{\R^2}
e^{iN \phi(x, t; \xi)} a(x, t; \xi) f(\xi) d\xi,
\end{equation}
where $N\in \R$ is a large real number, and $a(x, t; \xi)$ is a smooth function supported in a small neighborhood of the origin. More precisely, we will pick $\epsilon_{\phi}>0$ to be a sufficiently small constant depending on $\phi$, and assume that $a(x, t; \xi)$ is supported inside 
\begin{equation}
\B^{(2)}_{\epsilon_{\phi}}\times \B^{(1)}_{\epsilon_{\phi}}\times \B^{(2)}_{\epsilon_{\phi}}.
\end{equation}
Here we use $\B^{(n)}_{\epsilon}$ to denote the ball in $\R^n$ of radius $\epsilon$ centered at the origin. 
The estimates we are interested in proving are of the form 
\begin{equation}\label{260503e1_2}
\Norm{
T_N^{(\phi)} f
}_{L^p(\R^3)}\lesim_{\phi, a, p, \epsilon} N^{-\frac{3}{p}+\epsilon} \|f\|_{L^p(\R^2)},
\end{equation}
for every $\epsilon>0$ and every $N\ge 1$, and for a range of exponents $p$ that is as large as possible. \\

H\"ormander \cite{Hor73} and Bourgain \cite{Bou91} observed that to study the estimate \eqref{260503e1_2}, one can apply elementary changes of variables, and only need to consider phase functions $\phi$ of the form 
\begin{equation}
\phi(x, t; \xi)= \inn{x}{\xi}+ t\inn{A\xi}{\xi}+
O\pnorm{
|t||\xi|^3+ |\bfx|^2 |\xi|^2
}, \ \ \bfx:=(x, t),
\end{equation}
for some non-degenerate $2\times 2$ matrix $A$; these are called \underline{normal forms} at the origin. If the matrix $A$ is positive definite, we say that $\phi$ is elliptic.  \\

The geometry of \underline{characteristic curves} (see the definition below) plays a central role in the study of H\"ormander-type operators. 

\begin{definition}[Characteristic curves and $\phi$-Kakeya sets]\label{260601defi1_1}
For $\xi\in \B^{(2)}_{\epsilon_{\phi}}$ and $\bfx=(x, t)\in \B^{(3)}_{\epsilon_{\phi}}$, define 
\begin{equation}
\Gamma^{(\phi)}_{\xi}(\bfx):=
\set{
\bfx'\in \B^{(3)}_{2\epsilon_{\phi}}: \nabla_{\xi}\phi(\bfx'; \xi)=\nabla_{\xi}\phi(\bfx; \xi)
}.
\end{equation}
If it is clear from the context which $\phi$ is involved, we will simplify the notation $\Gamma^{(\phi)}_{\xi}(\bfx)$ to $\Gamma_{\xi}(\bfx)$. If $\bfx=(x, 0)$, then we simply write $\Gamma_{\xi}(x)$ instead of $\Gamma_{\xi}(\bfx)$.

A Borel set $E\subset \R^3$ is called a $\phi$-Kakeya set if for every $\xi\in \B^{(2)}_{\epsilon_{\phi}}$ there exists $x\in \B^{(2)}_{\epsilon_{\phi}}$ such that $\Gamma_{\xi}(x)\subset E$. 
\end{definition}

To study finer properties for the oscillatory integral operator $T_N^{(\phi)}$ and dimensions of $\phi$-Kakeya sets, Bourgain \cite{Bou91} introduced the following notion, which was called Bourgain's condition in Guo, Wang and Zhang \cite{GWZ24}. 

\begin{definition}[Bourgain's condition, \cite{Bou91}]
Given a phase function $\phi(x, t; \xi)$ and a fixed point $(x_0, t_0; \xi_0)$. We write the phase function 
\begin{equation}
\phi(x+ x_0, t+t_0; \xi+ \xi_0)
\end{equation}
in its normal form at the origin $x=\xi=0, t=0$, say $\widetilde{\phi}(x, t; \xi)$. We say that $\phi$ satisfies Bourgain's condition at $(x_0, t_0; \xi_0)$ if 
\begin{equation}
\nabla^2_{\xi} \partial^2_t \widetilde{\phi}(0, 0; 0) = c(x_0, t_0; \xi_0) \nabla^2_{\xi} \partial_t \widetilde{\phi}(0, 0; 0) 
\end{equation}
for some scalar $c(x_0, t_0; \xi_0)$. We say that $\phi$ satisfies Bourgain's condition along the characteristic curve $\Gamma_{\xi}(\bfx)$ if it satisfies Bourgain's condition at every point on this curve. 
\end{definition}

Examples for which Bourgain's condition holds include the Fourier extension operator, the Bochner-Riesz operator, the Schr\"odinger propagator, the resolvent of the standard Laplace operator in the Euclidean space, and so on. We refer to \cite{GWZ24} for a detailed discussion on the Bourgain's condition for these operators and more other operators. 

In a certain sense, operators satisfying Bourgain's condition admit certain ``best possible" behavior. For instance, it is proven in \cite{GWZ24} that operators with Bourgain's condition satisfy the same polynomial Wolff axioms as the ones for the standard Fourier extension operators. Moreover, Nadjimzadah \cite{Nad25} proved that for any H\"ormander-type oscillatory integral operator satisfying Bourgain's condition, the sticky case of the corresponding curved Kakeya conjecture reduces to the sticky case of the classical Kakeya conjecture.\\

After Bourgain's work \cite{Bou91}, Minicozzi and Sogge \cite{MS97}, Sogge \cite{Sog99}, Bourgain and Guth \cite{BG11}, Guth, Hickman and Iliopoulou \cite{GHI19} and Hickman and Iliopoulou \cite{HI22} studied more H\"ormander-type operators that have less favorable properties compared with those satisfying Bourgain's condition; we also refer to the references therein for more results. Recently, \cite{GWZ24}, \cite{DGGZ24} and \cite{CGG25}, inspired by the previously mentioned works, tried to push this research direction further; in particular, they tried to classify all H\"ormander-type operators. While doing so, some interesting applications (say \cite{BGH26}) were found. \\

To describe these developments, we need to introduce more notation. 
For a fixed $x_0\in\R^2$, let $X(t; x_0; \xi)$ be the unique solution to 
\begin{equation}\label{260502e1_1}
\nabla_{\xi} \phi(X(t; x_0; \xi), t; \xi)= \nabla_{\xi} \phi(x_0, 0; \xi), \ \forall t.  
\end{equation}
In other words, the characteristic curve $\Gamma_{\xi}(x_0)$ is given by the trajectory of 
\begin{equation}
(X(t; x_0; \xi), t).
\end{equation}
We introduce the notation 
\begin{equation}\label{260504e1_10}
W_{ij}(t; x_0; \xi):= (\partial_{\xi_i}\partial_{\xi_j}\phi)(
X(t; x_0; \xi), t; \xi
)- (\partial_{\xi_i}\partial_{\xi_j}\phi)(
x_0, 0; \xi
).
\end{equation}
Moreover, define 
\begin{equation}
\mc{W}(t; x_0; \xi):=
\det
\begin{bmatrix}
W_{11}(t; x_0; \xi), & W_{12}(t; x_0; \xi)\\
W_{21}(t; x_0; \xi), & W_{22}(t; x_0; \xi)
\end{bmatrix}.
\end{equation}
Consider the $4\times \infty$ matrix
\begin{equation}\label{260309e1_4}
\mathfrak{W}(x_0; \xi):= 
\begin{bmatrix}
\frac{\partial}{\partial t} \mc{W}(0; x_0; \xi), & \frac{\partial^2}{\partial t^2} \mc{W}(0; x_0; \xi), & \dots, & \frac{\partial^k}{\partial t^k} \mc{W}(0; x_0; \xi), & \dots\\
\frac{\partial}{\partial t} W_{11}(0; x_0; \xi), & \frac{\partial^2}{\partial t^2} W_{11}(0; x_0; \xi), & \dots, & \frac{\partial^k}{\partial t^k} W_{11}(0; x_0; \xi), & \dots\\
\frac{\partial}{\partial t} W_{12}(0; x_0; \xi), & \frac{\partial^2}{\partial t^2} W_{12}(0; x_0; \xi), & \dots, & \frac{\partial^k}{\partial t^k} W_{12}(0; x_0; \xi), & \dots\\
\frac{\partial}{\partial t} W_{22}(0; x_0; \xi), & \frac{\partial^2}{\partial t^2} W_{22}(0; x_0; \xi), & \dots, & \frac{\partial^k}{\partial t^k} W_{22}(0; x_0; \xi), & \dots
\end{bmatrix}
\end{equation}
Let $\mathfrak{W}^-(x_0; \xi)$ be the $3\times \infty$ matrix consisting of the last three rows of $\mathfrak{W}(x_0; \xi)$. Moreover, let $\mathfrak{W}_{k}(x_0; \xi)$ and $\mathfrak{W}^-_{k}(x_0; \xi)$ be the first $k$ columns of $\mathfrak{W}(x_0; \xi)$ and $\mathfrak{W}^-(x_0; \xi)$, respectively. 

\begin{definition}\label{260205defi1_1}
Let $1\le r\le 4$ be an integer.
\begin{enumerate}
\item We say that the phase function $\phi(x, t; \xi)$ satisfies rank-$r$ condition at $(x_0; \xi)$ if  the matrix $\mathfrak{W}^-(x_0; \xi)$ has rank $r-1$ and the matrix $\mathfrak{W}(x_0; \xi)$ has rank $r$. 

\item We say that the phase function $\phi(x, t; \xi)$ satisfies rank-$(\le r)$ condition at $(x_0; \xi)$ if  the matrix $\mathfrak{W}^-(x_0; \xi)$ has rank $r'$ for some $r'\le r-1$ and the matrix $\mathfrak{W}(x_0; \xi)$ has rank $r'+1$. 
\end{enumerate}
Moreover, we say that the phase function $\phi(x, t; \xi)$ satisfies the rank-$r$ (rank-$(\le r)$, respectively) condition if Item 1 (Item 2, respectively) holds at every point. 
\end{definition}

It is elementary to see that the rank of $\mathfrak{W}$ is at least two at any given point. To see why the first row in \eqref{260309e1_4} plays a special role in Definition \ref{260205defi1_1}, let us compare the example of the standard Fourier extension operator, 
\begin{equation}
    \phi_1(x, t; \xi)= \inn{x}{\xi}+ \frac{1}{2} t|\xi|^2,
\end{equation}
with Bourgain's example in \cite{Bou91} 
\begin{equation}
    \phi_2(x, t; \xi)= \inn{x}{\xi}+ t\xi_1 \xi_2+ \frac{1}{2} t^2 \xi_1^2,
\end{equation}
where $x=(x_1, x_2), \xi=(\xi_1, \xi_2)$. Elementary calculation shows that the matrix $\mathfrak{W}$ always has rank two, for both $\phi_1$ and $\phi_2$. However, we expect \eqref{260503e1_2} to hold for all $p\ge 3$ (the Fourier restriction conjecture) for the phase function $\phi_1$, while Bourgain \cite{Bou91} showed that \eqref{260503e1_2} fails for all $p<4$ for the phase function $\phi_2$. 

Let us also point out here that Definition \ref{260205defi1_1} already implicitly appeared in \cite{CGG25}, and earlier forms of this definition can be found in Bourgain \cite{Bou91} and Sogge \cite{Sog99}. 

\begin{theorem}\label{260518theorem1_4}
    Let $r\in \{2, 3, 4\}$ and $k\in \N$ be two integers. Let $\phi$ be an elliptic phase function. Assume that $\phi$ satisfies the rank$-r$ condition at every point, and assume that 
    \begin{equation}
        \rank \mathfrak{W}_k(x_0; \xi)= r, \ \forall x_0, \xi.
    \end{equation}
    Then the estimate \eqref{260503e1_2} holds for all
    \begin{equation}\label{260518e1_14}
        p\ge \frac{10}{3}-\epsilon_{k},
    \end{equation}
    where 
    \begin{equation}\label{260518e1_15}
        \epsilon_{k}:=\frac{1}{9k-6}.
    \end{equation}
\end{theorem}

The case $r=4$ of Theorem \ref{260518theorem1_4} was contained in \cite{DGGZ24}; the cases $r=2, 3$ were not stated in  \cite{DGGZ24} but essentially the same argument\footnote{
In particular, exactly the same numerology as that in the proof of the polynomial Wolff axiom in Subsection 4.3 in \cite{DGGZ24} works. 
} works, and therefore we will not repeat the proof. For the case $r=2$,  \cite{GWZ24} already obtained a  range better than \eqref{260518e1_14}.

Theorem \ref{260518theorem1_4} may not hold anymore if the rank-$r$ condition is replaced by rank-$(\le r)$ condition, and this was shown already in Minicozzi and Sogge \cite{MS97} and Sogge, Xi and Xu \cite{SXX18}. Moreover, Remark 5.4 in \cite{CGG25} provided an example showing why the key ingredient in the argument in \cite{DGGZ24}, which is the polynomial Wolff axiom (see \cite[Section 4.3]{DGGZ24}),  fails to work under the rank-$(\le r)$ condition.\\

Theorem \ref{260518theorem1_4} concerns H\"ormander-type oscillatory integrals. If we let $E_{\phi}\subset \R^3$ be a $\phi$-Kakeya set, then one can essentially follow the argument in Hickman, Rogers and Zhang \cite{HRZ22}, as was done in \cite[Section 6]{GLX25}, and obtain under the same assumption as in Theorem \ref{260518theorem1_4} that 
\begin{equation}\label{260518e1_16}
    \dim(E_{\phi})\ge 
    \frac{4k-1}{2k-1}.
\end{equation}
Note that when $k=4$ (which forces $r=2, 3, 4$), this is equal to $2+ 1/7$; when $k=3$ (which forces $r=2, 3$), this is equal to $2+1/5$; when $k=2$ (which forces $r=2$), this is equal to $2+1/3$.

\subsection{Carleson-Sj\"olin operators}

We recall the setup of Carleson-Sj\"olin operators on manifolds, in Subsection 1.3 in \cite{DGGZ24}. We work with a smooth metric $\{g_{ij}(\bfx)\}_{1\le i, j\le 3}$ defined on a small neighborhood of the origin $0\in \R^3$. The manifold is denoted by $\mc{M}$. Let $\varepsilon_{\mc{M}}>0$ be a sufficiently small constant depending on $\mc{M}$. Take $\varepsilon\le \varepsilon_{\mc{M}}/10$ and let $\gamma: (-2\varepsilon, 2\varepsilon)\to \mc{M}$ be a geodesic with unit speed. Without loss of generality, we make a normalization so that 
\begin{equation}\label{260502e1_5}
    \gamma(0)=(0, 0, 0), \ \ \gamma(s)=(0, 0, s), \forall s. 
\end{equation}
Consider the (reduced) Carleson-Sj\"olin operator, given by 
\begin{equation}\label{260518e1_18}
    R_{\varepsilon, N} f(\bfx):= \int_{
    |y|\le \varepsilon/10
    } e^{
    i N \dist(\bfx, (y, \varepsilon)) 
    }
    a(\bfx) f(y) 
    d y,
\end{equation}
where $\bfx=(x_1, x_2, x_3)$, $y=(y_1, y_2)$, and $a$ is a smooth bump function supported on $|\bfx|\le \varepsilon/10$. It is elementary (see Subsection 5.2 in \cite{DGGZ24}) that 
\begin{equation}\label{260603e1_21zzz}
    \dist_{\epsilon}(\bfx, y):= \dist(\bfx, (y, \epsilon))
\end{equation}
satisfies H\"ormander's non-degeneracy condition. Take
\begin{equation}\label{260601e1_20}
    \phi(\bfx; y):= \dist_{\epsilon}(\bfx, y).
\end{equation}
We say that $E\subset \mathcal{M}$ is a Kakeya set on the manifold $\mathcal{M}$ if it is a $\phi$-Kakeya set. \\

It is proven in \cite[Lemma 3.2]{DGGZ24} that characteristic curves for the phase function are all given by geodesics. More precisely, let $\gamma$ be a geodesic passing through $(y, \epsilon)$. Then $\nabla_y \phi(\bfx; y)$ stays constant when $\bfx$ moves along $\gamma$. This property explains the geometric meaning of Kakeya sets on manifolds.

\subsection{Various curvatures}

\begin{definition}[Chaotic curvature]\label{260505defi1_6}
Let $\gamma$ be a geodesic parametrized by arclength with $\gamma(0)\in \mc{M}$. Take a unit vector $X(0)\in T_{\gamma(0)}\mc{M}$ with 
\begin{equation}
X(0)\perp \dot{\gamma}(0).
\end{equation}
Let $X(t)\in T_{\gamma(t)}\mc{M}$ be the parallel transport of $X(0)$ along $\gamma$. Denote 
\begin{equation}
Y(t):= \bar{\mathrm{Ric}}(X(t)),
\end{equation}
and let $Y^{\perp}(t)$ be the projection of $Y(t)$ to the orthogonal complement of the space spanned by $\dot{\gamma}(t)$ and $X(t)$. 
\begin{enumerate}
\item We say that the manifold $\mc{M}$ satisfies the chaotic curvature condition of order $\le k$ at $\gamma(0)$ along $\gamma$ if 
\begin{equation}\label{251106e1_3}
|Y^{\perp}(0)|+\dots+ 
|
\nabla^k_{\dot{\gamma}}Y^{\perp}(0)
|\neq 0
\end{equation}
for every choice of $X(0)$. 

\item We say that the manifold $\mc{M}$ satisfies the chaotic curvature condition of order $\le k$ if \eqref{251106e1_3} is satisfied for every point and every geodesic $\gamma$ passing through the point
\end{enumerate}
\end{definition}

Chaotic curvatures of order $\le 1$ along geodesics were introduced by Sogge \cite{Sog99} when he was studying Kakeya problems on Riemannian manifolds. In Sogge's paper \cite{Sog99}, chaotic curvature of order $\le 1$ was simply called chaotic curvature. Moreover, he used the name ``variably curved manifold" if a manifold $\mc{M}$ satisfies the chaotic curvature condition of order $\le 1$ (along every geodesic).  Definition \ref{260505defi1_6} generalizes Sogge's chaotic curvature from order $\le 1$ to order $\le k$ for every $k$. \\

Let us also remark here that in \cite{Sog99}, Sogge  studied Kakeya problems not only on manifolds satisfying the chaotic curvature condition, but also on manifolds (of dimension three) that have constant sectional curvature. Later, Xi \cite{Xi17} extended Sogge's result for three dimensional manifolds of constant sectional curvature to every higher dimension.

Manifolds of constant sectional curvature admit certain ``best" possible behavior in the study of Kakeya problems. For instance, it was proven in \cite{GWZ24} and \cite{DGGZ24} that the distance function $\dist_{\epsilon}(\bfx, y)$ (given in \eqref{260603e1_21zzz}) satisfies Bourgain's condition for every $\epsilon, \bfx$ and $y$ if and only if the manifold has constant sectional curvature (see also Corollary \ref{260509coro2_2} below).  

On the other hand,  manifolds satisfying the chaotic curvature condition have ``least possible" symmetries\footnote{In comparison, manifolds of constant sectional curvature have ``most possible" symmetries.}, and this explains the name ``chaotic" by \cite{Sog99}. At the end of \cite{Sog99}, Sogge proposed to study Kakeya problems on general (locally) symmetric spaces, which, in terms of symmetries, lie between manifolds satisfying the chaotic curvature condition and manifolds of constant sectional curvature. We will see that this is one main goal of the current paper, see for instance Corollary \ref{260509coro2_3} and Corollary \ref{260507theorem1_11}.

\begin{definition}[$k$-exceptional, Lytchak and Petrunin \cite{LP22}]\label{def:k exceptional}
Let $\gamma$ be a geodesic parametrized by arc-length. Let $k\ge 2$. Define the Jacobi operators of order $k$: 
\begin{equation}
\mc{R}_{\gamma(0)}^k: v\mapsto \nabla_{\gamma'(0)}^{k-2}\mathcal{R}(v,\gamma'(0))\gamma'(0).
\end{equation}
The geodesic $\gamma$ on the three dimensional manifold $\mc{M}$ is called $k$-exceptional if there exists a two dimensional subspace $V\subset T_{\gamma(0)}\mc{M}$ with $\gamma'(0)\in V$ such that 
\begin{equation}
    \mc{R}^{k'}_{\gamma(0)}: V\rightarrow V
\end{equation}
for every $2\le k'\le k$.  Moreover, a manifold $\mc{M}$ is called $k$-exceptional if there exists a geodesic that is $k$-exceptional. 
\end{definition}

Lytchak and Petrunin \cite{LP22} introduced the notion  $k$-exceptional
 in their study of convex sets in generic manifolds. In particular, they introduced $k$-exceptional as a finite-order curvature obstruction for the existence of totally geodesic sub-manifolds.

\begin{definition}[Berndt and Vanhecke, \cite{BV92}]\label{260504defi1_4}
  Let $\gamma(t)$ be a geodesic on $\mc{M}$. Let $E_1(0), E_2(0)$ be two perpendicular unit vectors in the tangent space $T_{\gamma(0)}\mc{M}$ satisfying $E_i(0)\perp \gamma'(0)$, $i=1, 2$. Let $E_i(t)$ be the parallel transport of $E_i(0)$ along $\gamma$. Let $\mathcal{R}$ be the Riemannian curvature tensor. Define 
  \begin{equation}\label{260504e1_16}
      R_{ij}(t):= \langle \mathcal{R}(E_i(t), \gamma'(t))\gamma'(t), E_j(t)\rangle, \ \ 1\le i, j\le 2.
  \end{equation}
  Moreover, define 
  \begin{equation}\label{260504e1_17}
      R(t):=\begin{bmatrix}
          R_{11}(t) & R_{12}(t)\\
          R_{21}(t) & R_{22}(t)
      \end{bmatrix}
  \end{equation}
  We say that the manifold $\mc{M}$ is a $\mathfrak{P}$-space if $R(t)$ is simultaneously diagonalizable along $\gamma$ for every $\gamma$. 
\end{definition}

The notation $\mathfrak{P}$ is for parallel. The notion of $\mathfrak{P}$-spaces was introduced by Berndt and Vanhecke \cite{BV92} to study generalizations of locally symmetric spaces, and to study classifications of Riemannian manifolds. \\

It is not difficult to see that all these three definitions are closely connected. In our main theorem (Theorem \ref{260509theorem2_1} below), we will make these connections more clear; roughly speaking, all these three notions of curvatures are equivalent.

\section{Statement of main results}\label{260601section2}

We calculate the matrix \eqref{260309e1_4} for $\dist_{\epsilon}(\bfx, y)$. More precisely, given a geodesic $\gamma$, say the one given by \eqref{260502e1_5}, if we replace $\phi$ in \eqref{260502e1_1} by the distance function $\dist_{\epsilon}(\bfx, y)$, then 
\begin{equation}
    X(t, 0; 0)=0, \ \forall t. 
\end{equation}
With this preparation, one can compute the $4\times\infty$ matrix \eqref{260309e1_4} for the distance function $\dist_{\epsilon}(\bfx, y)$ along the given geodesic $\gamma$. For future reference, let us write this matrix as 
\begin{equation}\label{260509e2_2}
    \mathfrak{W}:=
    \left[\begin{array}{ccccc}
\partial_s \mathcal{W}(0, \epsilon) & \partial_s^2 \mathcal{W}(0, \epsilon) & \dots & \partial_s^k \mathcal{W}(0, \epsilon) & \dots \\
\partial_s W_{11}(0, \epsilon) & \partial_s^2 W_{11}(0, \epsilon) & \dots & \partial_s^k W_{11}(0, \epsilon) & \dots \\
\partial_s W_{12}(0, \epsilon) & \partial_s^2 W_{12}(0, \epsilon) & \dots & \partial_s^k W_{12}(0, \epsilon) & \dots \\
\partial_s W_{22}(0, \epsilon) & \partial_s^2 W_{22}(0, \epsilon) & \dots & \partial_s^k W_{22}(0, \epsilon) & \dots
\end{array}\right].
\end{equation}
We use $\mathfrak{W}^-$ to denote the last three rows of $\mathfrak{W}$ and $\mathfrak{W}_k$ to denote the first $k$ columns of the matrix $\mathfrak{W}$.  Moreover, we will recycle the setup in \eqref{260504e1_16} and \eqref{260504e1_17}.

\begin{theorem}\label{260601theorem2_1}
    The first row in \eqref{260509e2_2} is not a linear combination of the second, third and fourth rows. 
\end{theorem}

Theorem \ref{260601theorem2_1} is relatively easy to prove. It reinforces the fact that the first row plays a particular role, as already discussed below Definition \ref{260205defi1_1}. As a consequence of Theorem \ref{260601theorem2_1}, we know that for Riemannian distance functions, the matrix \eqref{260509e2_2} satisfies the rank-$r$ condition  if and only if it has rank $r$.

\begin{theorem}\label{260509theorem2_1}
    Given a geodesic $\gamma$. Assume that we are in the setup in \eqref{260502e1_5}, \eqref{260504e1_16} and \eqref{260504e1_17}. 
    \begin{enumerate}
        \item[(1)]  The matrix \eqref{260509e2_2} has rank two for all small enough $\epsilon$ if and only if 
        \begin{equation}
            R(s) \text{ is simultaneously diagonalizable} \ \forall s,
        \end{equation}
        and 
        \begin{equation}
            R_{11}(s)=R_{22}(s), \ \ \forall s;
        \end{equation}
        moreover, in this case the first two columns already have rank two. 
        \item[(2)] The matrix \eqref{260509e2_2} has rank three for all small enough $\epsilon$ if and only if 
        \begin{equation}
            R(s) \text{ is simultaneously diagonalizable} \ \forall s,
        \end{equation}
        and 
        \begin{equation}
            R_{11}\not\equiv_s R_{22};
        \end{equation}
        moreover, in this case the first $k$ ($k\ge 3$) columns of the $4\times \infty$ matrix has rank three if and only if 
        \begin{equation}
            R^{(m)}_{11}(0)\neq R^{(m)}_{22}(0)
        \end{equation}
        for some $0\le m\le k-3$; 
        \item[(3)] The matrix \eqref{260509e2_2} has rank four for all small enough $\epsilon$ if and only if 
        \begin{equation}
            R(s) \text{ is not simultaneously diagonalizable in\ }  s;
        \end{equation}
        moreover, for every $k\ge 4$ the following statements are equivalent: 
          \begin{enumerate}
    \item[(3.1)] The first $k$ columns of the matrix \eqref{260509e2_2} have rank $4$ for sufficiently small $\epsilon$;
    \item[(3.2)] There exists $\epsilon>0$ such that the first $k$ columns of the matrix \eqref{260509e2_2} has rank $4$;
    \item[(3.3)] The matrices $R(0), R'(0), \dots, R^{(k-3)}(0)$ cannot be diagonalized simultaneously; 
    \item[(3.4)] The manifold $\mc{M}$ satisfies the chaotic curvature condition of order $\le (k-3)$ at $\gamma(0)$ along $\gamma$; 
   	\item[(3.5)]  The geodesic $\gamma$ is not $(k-1)$-exceptional. 
    \end{enumerate}
    \end{enumerate}
\end{theorem}

We summarize the above result in the following table.  In the following table, we abbreviate ``simultaneously diagonalizable'' to (SD). Moreover, all the quantities $R, R_{ij}$ and their derivatives are evaluated at $s=0$. 
\begin{table}[htbp]
    \centering
    \begin{tabular}{|c|}
        \hline
        $(R^{(m)})_{m=0}^{\infty}$ is SD, $(R_{11}^{(m)})_{m=0}^{\infty}=(R_{22}^{(m)})_{m=0}^{\infty}$ \\
        \hline
        $(R^{(m)})_{m=0}^{\infty}$ is SD,  
        $(R_{11}^{(m)})_{m=0}^{k-3}\neq (R_{22}^{(m)})_{m=0}^{k-3}$
        \\
        \hline
        $(R^{(m)})_{m=0}^{k-3}$ is not SD\\
        \hline
    \end{tabular}
    \quad $\Leftrightarrow$ \quad
    \begin{tabular}{|c|}
        \hline
        $\rank\ \mathfrak{W}=2$ \\
        \hline
        $\rank\ \mathfrak{W}=\rank\ \mathfrak{W}_{k}=3$\\
        \hline
        $\rank\ \mathfrak{W}=\rank\ \mathfrak{W}_{k}=4$\\
        \hline
    \end{tabular}
    
    \caption{Equivalence between local curvature conditions and rank-order conditions}
\end{table}

The first item in Theorem \ref{260509theorem2_1}, combined with the results in \cite{DGGZ24}, has the following corollary. 

\begin{corollary}\label{260509coro2_2}
    Let $\mathcal{M}$ be a three dimensional manifold. The following statements are equivalent:
    \begin{enumerate}
         \item[(1)] the manifold $\mc{M}$ has constant sectional curvature, 
    \item[(2)] the distance function $\dist_{\epsilon}(\bfx, y)$ satisfies Bourgain's condition for every $\bfx, y$ and $\epsilon$,
    \item[(3)] the distance function $\dist_{\epsilon}(\bfx, y)$ satisfies rank-$2$ condition for every $\bfx, y$ and $\epsilon$.
    \end{enumerate}
\end{corollary}

The second item in Theorem \ref{260509theorem2_1} has the following corollary. 
\begin{corollary}\label{260509coro2_3}
    Let $\mc{M}$ be a three dimensional analytic manifold. Then $\mc{M}$ is a $\mathfrak{P}$-space if and only if $\dist_{\epsilon}(\bfx, y)$ satisfies rank-$(\le 3)$ condition at every point and for every $\epsilon$. 
\end{corollary}

So far we have finished our discussion on rank-$2$ conditions and rank-$(\le 3)$ conditions. Next, we will discuss rank-$4$ condition.

That $(3.3)$ implies (3.1) when $k=4$ has been obtained in \cite{DGGZ24}. \footnote{Note that when $k=4$, item (3.3) is precisely Sogge's chaotic curvature condition \cite{Sog99}, which we call chaotic curvature condition of order $\le 1$ in Definition \ref{260505defi1_6}. Moreover, in \cite{DGGZ24} it is proven that under Sogge's chaotic curvature condition, the determinant of the leading $4\times 4$ matrix of \eqref{260309e1_4} does not vanish if $\epsilon>0$ is chosen small enough, see Theorem 2.5 there. } The proof there is via direct (which means brute-force) calculations, and we find it very complicated, if not impossible at all, to generalize to larger $k$ and even to proving that (3.1) implies (3.3). 

Some of the calculations in \cite{DGGZ24} will be reviewed when we prove Theorem \ref{260509theorem2_1}, see for instance the beginning of Subsection \ref{260505sub4_1}. We will find that these calculations are still somewhat useful: They will help us accept certain patterns in the recursion relations for general $k$, see for instance Proposition \ref{260505prop4_5}, and will also help us understand Theorem \ref{260509theorem2_1} better. 

Let us be more precise here about what we mean by understanding Theorem \ref{260509theorem2_1} better. In \cite{DGGZ24}, the authors first calculated the $4\times \infty$ matrix \eqref{260509e2_2}. 
 In order to prove that (3.3) implies (3.1) when $k=4$, the authors expanded the determinant of the leading $4\times 4$ matrix in Taylor's series in the $\epsilon$ variable, and obtain that this determinant is equal to 
 \begin{equation}\label{260505e1_24a}
 -\frac{1}{144}\pnorm{
 (\mf{g}_{33, 11}-\mf{g}_{33, 22}) \mf{g}_{33, 123}
 - (\mf{g}_{33, 113}-\mf{g}_{33, 223}) \mf{g}_{33, 12}
 } \epsilon^5+ O(\epsilon^6), 
 \end{equation}
 where 
\begin{equation}\label{260505e1_25a}
R(0)=
-\frac12
\begin{bmatrix}
\mf{g}_{33, 11} & \mf{g}_{33, 12}\\
\mf{g}_{33, 12} & \mf{g}_{33, 22}
\end{bmatrix}
\end{equation}
and 
\begin{equation}\label{260505e1_26a}
R'(0)=
-\frac12
\begin{bmatrix}
\mf{g}_{33, 113} & \mf{g}_{33, 123}\\
\mf{g}_{33, 123} & \mf{g}_{33, 223}
\end{bmatrix},
\end{equation}
see (3.109), (3.95) and (3.96) in \cite{DGGZ24}. Note that $R(0)$ and $R'(0)$ commuting means exactly that the coefficient of $\epsilon^5$ in \eqref{260505e1_24a} vanishes, and this proves that (3.3) implies (3.1).

Let us see what (3.1) implying (3.3) means, which is equivalent to saying that if $R(0)$ and $R'(0)$ commute, then the determinant of the leading $4\times 4$ matrix of  \eqref{260309e1_4} vanishes for every $\epsilon>0$ (by analyticity), in other words, all the Taylor coefficients in the expansion \eqref{260505e1_24a} vanish. However, the calculation in \eqref{260505e1_24a} only tells us that the leading coefficient vanishes under the assumption that $R(0)$ and $R'(0)$ commute, and we still need to prove that 
\begin{equation}
(\mf{g}_{33, 11}-\mf{g}_{33, 22}) \mf{g}_{33, 123}
 - (\mf{g}_{33, 113}-\mf{g}_{33, 223}) \mf{g}_{33, 12}
\end{equation}
is a factor of every Taylor coefficient; this is what we proved in Theorem \ref{260509theorem2_1}. In other words, if the coefficient of the lowest order monomial (which is $\epsilon^5$ in the current case) vanishes, then all other coefficients will vanish as well. 

This phenomenon still looks quite surprising to us, and it is not clear to us whether any similar phenomenon has been observed in the literature. This phenomenon seems to  suggest that by only knowing some ``finite order" curvature information, we may conclude curvature information for  ``infinite orders", which, combined with analyticity, will determine certain local geometry completely. \\

In the same paper \cite{BV92} that Berndt and Vanhecke introduced $\mathfrak{P}$-spaces, they also introduced the notion of $\mathfrak{C}$ spaces: Under the same setting as in Definition \ref{260504defi1_4}, we say that $\mathcal{M}$ is a $\mathfrak{C}$-space if the eigenvalues of $R(t)$ are constant in $t$, for every geodesic $\gamma$. Moreover, they proved that 
\begin{equation}
    (\mathfrak{P}-\text{spaces})\cap (\mathfrak{C}-\text{spaces})= \text{locally symmetric spaces}.
\end{equation}
Let us bring the notion of $\mathfrak{C}$ spaces into the whole picture of classifying curvature conditions described by Theorem \ref{260509theorem2_1}. 

\begin{corollary}\label{260507theorem1_11}
    Fix a geodesic $\gamma$, and assume that the eigenvalues of the matrix $R(s)$ are constant along $\gamma$. 
    We have that 
    \begin{enumerate}
        \item[(1)] The $4\times \infty$ matrix has rank four for all small $\epsilon>0$ if and only if 
        \begin{equation}
            \mathrm{dim} \pnorm{\operatorname{span}\{I, R(0), R'(0), R''(0), \dots\}
            } =3,
        \end{equation}
        if and only if $R(s)$ is not constant in $s$; 
        \item[(2)] The $4\times \infty$ matrix has rank three for all small $\epsilon>0$ if and only if 
        \begin{equation}
            \mathrm{dim} \pnorm{\operatorname{span}\{I, R(0), R'(0), R''(0), \dots\}
            } =2,
        \end{equation}
        if and only if $R(s)$ is constant in $s$ and 
        \begin{equation}
            \mathrm{dim} \pnorm{\operatorname{span}\{I, R(0)\}
            } =2;
        \end{equation}
        \item[(3)] The $4\times \infty$ matrix has rank two for all small $\epsilon>0$ if and only if 
        \begin{equation}
            \mathrm{dim} \pnorm{\operatorname{span}\{I, R(0), R'(0), R''(0), \dots\}
            } =1,
        \end{equation}
        if and only if $R(s)$ is constant in $s$ and 
        \begin{equation}
            \mathrm{dim} \pnorm{\operatorname{span}\{I, R(0)\}
            } =1.
        \end{equation}
    \end{enumerate}
\end{corollary}

Fix $1\le r\le 4$ and $k\in \N$. 
    We say that a three dimensional analytic manifold $\mathcal{M}$ is 
    called an $(r, k)$-manifold if 
    \begin{equation}
        \mathrm{rank}\  \mathfrak{W}= \mathrm{rank} \ \mathfrak{W}_k=r,
    \end{equation}
    for every geodesic $\gamma$.  \\

As an immediate corollary of Theorem \ref{260509theorem2_1} and the discussions around \eqref{260518e1_16}, we obtain the following corollary. 

\begin{corollary}\label{260518corollary2_5}
    Fix $1\le r\le 4$ and $k\in \N$.  For a three dimensional analytic $(r, k)$-manifold, its Kakeya sets have dimension at least
    \begin{equation}\label{260518e2_19}
        2+\frac{1}{2k-1}.
    \end{equation}
    In particular, for $k\ge 4$, if a three dimensional analytic manifold satisfies the chaotic curvature condition of order $\le (k-3)$, which is the same as saying that the manifold is not $(k-1)$-exceptional, then its Kakeya sets have dimensions at least \eqref{260518e2_19}. 
\end{corollary}

Corollary \ref{260518corollary2_5} discusses Kakeya sets on manifolds satisfying the chaotic curvature condition of order $\le (k-3)$, or equivalently, manifolds that are not $(k-1)$-exceptional, for all $k\ge 4$. However, it is not even clear whether such manifolds exist. 
The authors in \cite{LP22} remarked that ``probably every Riemannian metric is $3$-exceptional". The theorem below confirms that this is indeed the case.

\begin{theorem}\label{existence}
Let $\mc{M}$ be an analytic manifold of dimension three. 
\begin{enumerate}
\item[(1)] There are no manifolds that satisfy the chaotic curvature condition of order $\le 1$, or equivalently, every manifold is $3$-exceptional; 
\item[(2)] One can find a manifold $\mc{M}_1$ satisfying the chaotic curvature condition of order $\le 2$, and so does every smooth perturbation of it; moreover, one can also find a manifold $\mc{M}_2$ failing the chaotic curvature condition of order $\le 2$, and so does every smooth perturbation of it; 
\item[(3)] A ``generic" manifold satisfies the chaotic curvature of order $\le 3$. More precisely, for a given metric $\{g_{ij}(\bfx)\}_{1\le i, j\le 3}$ defined on a small neighborhood of the origin $0\in \R^3$,  and every $\epsilon>0$ and $K$ sufficiently large ($K\ge 10$ is more than enough), there exists $\delta>0$ and a new metric $\{\widetilde{g_{ij}}(\bfx)\}_{1\le i, j\le 3}$ defined on $\B^{(3)}_{\delta}\subset \R^3$ such that 
\begin{equation}
    \|g_{ij}- \widetilde{g}_{ij}\|_{C^{K}(
    \B^{(3)}_{\delta}
    )}\le \epsilon, \ \ 1\le i, j\le 3
\end{equation}
and that the new metric $\{\widetilde{g_{ij}}(\bfx)\}_{1\le i, j\le 3}$ satisfies the chaotic curvature condition of order $\le 3$ on $\B^{(3)}_{\delta}$. 
\end{enumerate}

\end{theorem}

Item (1) of Theorem \ref{existence} says that  variably curved manifolds (defined below Definition \ref{260505defi1_6}) do not exist. A statement that is almost equivalent to Item (1) is that at every point on the manifold, we can find a geodesic $\gamma$  passing through that point such that the manifold fails the chaotic curvature condition of order $\le 1$ along $\gamma$ at that point. 

Of course one can easily find manifolds $\mc{M}$ and geodesics $\gamma$ on them such that $\mc{M}$ satisfies the chaotic curvature condition of order $\le 1$ at $\gamma(0)$ along $\gamma$, and therefore the more general version of Theorem 3.4 in \cite{Sog99}, which discusses chaotic curvatures of order $\le 1$ along a family of geodesics, is perfectly valid; indeed, this is the notion of chaotic curvature adopted by \cite{GLX25}, where the authors proved that a ``generic" three dimensional Riemannian manifold  satisfies Sogge's chaotic curvature condition of order $\le 1$, under the extra assumption that we are also allowed to restrict the directions of geodesics when defining what ``generic" means. \footnote{In other words, what ``generic" means in \cite{GLX25} is much more restrictive than in Theorem \ref{existence}.} \\

As a consequence of Theorem \ref{existence}, Theorem \ref{260509theorem2_1} and  the discussion around \eqref{260518e1_16}, we obtain immediately the following corollary. 

\begin{corollary}\label{260601coro2_7}
    For a  generic three dimensional analytic manifold, its Kakeya sets have dimension at least $2+1/11$. 
\end{corollary}

We also conjecture that for every three dimensional analytic manifold, there exists $\kappa>0$ depending on the manifold such that \eqref{260503e1_2} with $R_{\varepsilon, N} f$ (see \eqref{260518e1_18} for its definition) in place of  $T_N^{(\phi)}f$, holds for all 
\begin{equation}
    p\ge \frac{10}{3}-\kappa.
\end{equation}
If this were true, it would imply that for every three dimensional analytic manifold, there exists $\kappa'>0$ depending on the manifold such that its Kakeya sets have dimension at least $2+\kappa'$.

Moreover, Minicozzi and Sogge \cite{MS97} constructed a sequence of three dimensional analytic manifolds along which $\kappa\to 0$. In other words, $\kappa$ cannot be chosen to be a universal constant.

\section{Proof of corollaries}

The proofs for Corollary \ref{260509coro2_3}, Corollary \ref{260518corollary2_5} and Corollary \ref{260601coro2_7} are immediate, and we will skip their proofs. In this section, we will only present the proofs of Corollary  \ref{260509coro2_2} and Corollary \ref{260507theorem1_11}.

\subsection{Proof of Corollary \ref{260509coro2_2}}

Let us prove something more general. 
\begin{lemma}\label{260601lemma4_1}
Let $\phi$ be a phase function satisfying H\"ormander's non-degeneracy condition. 
    That $\phi$ satisfies rank-$2$ condition at $(x_0; \xi)$, is equivalent to saying that $\phi$ satisfies Bourgain's condition along $\Gamma_{\xi}(x_0)$. 
\end{lemma}

Recall that it is proven in \cite{DGGZ24} that the distance function $\dist_{\epsilon}(\bfx, y)$ satisfies Bourgain's condition for every $\bfx, y, \epsilon$ if and only if the manifold $\mathcal{M}$ has constant sectional curvature. This, combined with 
Lemma \ref{260601lemma4_1}, finishes the proof of Corollary \ref{260509coro2_2}.

\begin{proof}[Proof of Lemma \ref{260601lemma4_1}]
    This lemma is well-known to people in the community of curved Kakeya problems and H\"ormander's problems, but has never been written down. Let us include the proof here for completeness.  Without loss of generality, we assume that 
\begin{equation}
x_0=(0, 0), \ \ \xi=(0, 0),
\end{equation}
and that $\phi$ is of a normal form at the origin, that is, 
\begin{equation}\label{260503e3_2}
\phi(x, t; \xi)= \inn{x}{\xi}+ t\inn{A\xi}{\xi}+
O\pnorm{
|t||\xi|^3+ |\bfx|^2 |\xi|^2
}, \ \ \bfx:=(x, t),
\end{equation}
and $A$ is a diagonal non-degenerate $2\times 2$ matrix. Under this normalization, we know that 
\begin{equation}
X(t; x_0; \xi)= 0, \ \forall t,
\end{equation}
and 
\begin{equation}
W_{ij}(t; x_0; \xi)= 
\partial_{\xi_i}\partial_{\xi_j} \phi(0, t; 0). 
\end{equation}
From now on, to simplify notation, let us remove $x_0, \xi$ from our notation, and simply write $W_{ij}(t)$ for instance. \\

Let us write 
\begin{equation}
\nabla^2_{\xi} \phi(0, t; 0)= tA+ t^2 A_2+\dots.
\end{equation}
That $\phi$ satisfies Bourgain's condition along the characteristic curve $\Gamma_{\xi}(x_0)$ is equivalent to that 
\begin{equation}\label{260503e3_6}
A_i= c_i A, \ \ i=2, 3, \dots
\end{equation}
which can be proven by turning \eqref{260503e3_2} into a normal form at  $(0, t_0; 0)$ for every $t_0$. Moreover, it is elementary to see that \eqref{260503e3_6} is equivalent to that the matrix $\mathfrak{W}$ given by \eqref{260309e1_4} has rank two at $(x_0; \xi)$. This finishes the proof of the lemma.
\end{proof}

\subsection{Proof of Corollary \ref{260507theorem1_11}}

Let $c_1$ be the trace of $R(s)$, and $c_2$ its determinant. Let us first prove that 
\begin{equation}\label{260510e8_1}
    c_2- \frac{c_1^2}{4}\le 0.
\end{equation}
Note that 
\begin{equation}
    c_1= R_{11}+ R_{22}, \ \ c_2= R_{11}R_{22}- (R_{12})^2.
\end{equation}
We can compute 
\begin{equation}\label{260510e8_3}
    c_2- \frac{c_1^2}{4}=
    R_{11}R_{22}- (R_{12})^2- 
    \frac{(R_{11}+ R_{22})^2}{4}= 
    -(R_{12})^2-
    \frac{(R_{11}- R_{22})^2}{4}.
\end{equation}
This proves \eqref{260510e8_1}.\\

Under the assumption that the eigenvalues of $R(s)$ are constant in $s$, we know that both $c_1$ and $c_2$ are constants in $s$. Decompose 
\begin{equation}
    R(s)= \frac{c_1}{2} I+ R_-(s),
\end{equation}
where $R_-(s)$ is the trace-free part of $R(s)$. One can compute directly that 
\begin{equation}
    \det R_-(s)= c_2- \frac{c_1^2}{4}.
\end{equation}
Because of \eqref{260510e8_1}, we denote 
\begin{equation}
    c_2- \frac{c_1^2}{4}=-\mu^2\le 0, 
\end{equation}
for some $\mu\ge 0$.

\bigskip

Let us first prove Item (3). From \eqref{260510e8_3} it is easy to see that $\mu=0$ if and only if 
\begin{equation}
    R_{11}(s)\equiv R_{22}(s), \ \ R_{12}(s)\equiv 0,
\end{equation}
which is the same as saying that 
\begin{equation}
    R_-(s)=0, \ \forall s;
\end{equation}
in other words, the matrix $R(s)$ is a constant scalar matrix. 

Note that (3) has three statements: The first statement is 
\begin{equation}\label{260517e8_10}
    \text{the $4\times \infty$ matrix has rank two for all small $\epsilon$},
\end{equation}
the second statement is 
\begin{equation}\label{260517e8_11}
\operatorname{dim}\left(\operatorname{span}\left\{I, R(0), R^{\prime}(0), R^{\prime \prime}(0), \ldots\right\}\right)=1,
\end{equation}
and the third statement is 
\begin{equation}\label{260517e8_12}
    R(s) \text{\ is constant in\ } s, \ \operatorname{dim}(\operatorname{span}\{I, R(0)\})=1 .
\end{equation}
In Theorem \ref{260509theorem2_1} we already obtained that the matrix \eqref{260509e2_2} has rank two if and only if $R(s)$ is simultaneously diagonalizable, and after diagonalizing we have 
\begin{equation}
    R_{11}(s)=R_{22}(s), \ \forall s. 
\end{equation}
Recall that currently we are assuming that all eigenvalues of $R$ are constant. As a consequence, we obtain that the matrix \eqref{260509e2_2} has rank two if and only if $R(s)$ is a constant scalar matrix, that is, \eqref{260517e8_10} is equivalent to \eqref{260517e8_12}. 

Note that \eqref{260517e8_11} is the same as saying that $R(s)$ is a scalar matrix, and therefore by the assumptions that all eigenvalues of $R$ are constant, we can conclude that \eqref{260517e8_11} is the same as saying that $R(s)$ is a constant scalar matrix. 

So far we have proven that all the three statements \eqref{260517e8_10}, \eqref{260517e8_11} and \eqref{260517e8_12} are equivalent to that $R(s)$ is a constant scalar matrix, and moreover they are all equivalent to that $\mu=0$.\\

From now on let us assume that $\mu>0$. We are about to prove (1) and (2). It is elementary to see that to prove the entire Corollary \ref{260507theorem1_11}, we only need to prove one of the two statements (1) and (2). Here we will prove item (2). 

Item (2) contains three statements: The first statement is 
\begin{equation}\label{260517e8_10z}
    \text{the $4\times \infty$ matrix has rank three for all small $\epsilon$},
\end{equation}
the second statement is 
\begin{equation}\label{260517e8_11z}
\operatorname{dim}\left(\operatorname{span}\left\{I, R(0), R^{\prime}(0), R^{\prime \prime}(0), \ldots\right\}\right)=2,
\end{equation}
and the third statement is 
\begin{equation}\label{260517e8_12z}
    R(s) \text{\ is constant in\ } s, \ \operatorname{dim}(\operatorname{span}\{I, R(0)\})=2.
\end{equation}
That \eqref{260517e8_10z} is equivalent to \eqref{260517e8_12z} follows immediately from item (2) in Theorem \ref{260509theorem2_1}. That \eqref{260517e8_12z} implies \eqref{260517e8_11z} is trivial. It remains to prove that \eqref{260517e8_11z} implies \eqref{260517e8_12z}.

Note that \eqref{260517e8_11z} is the same as saying that 
\begin{equation}\label{260517e5_22z}
\operatorname{dim}\left(\operatorname{span}\left\{R_-(0), R_-^{\prime}(0), R_-^{\prime \prime}(0), \ldots\right\}\right)=1. 
\end{equation}
By the fact that $\mu>0$, we see that $R_-(0)$ is not the zero matrix. This, combined with \eqref{260517e5_22z}, implies that $R^{(m)}_-(0)$ is a multiple of $R_-(0)$, for every $m\ge 1$. We claim that 
\begin{equation}
    R^{(m)}_-(0)=0, \ \forall m\ge 1.
\end{equation}
Once this is proven, we see immediately that \eqref{260517e8_12z} holds.

\begin{claim}\label{260509claim8_1}
    Let $\mathcal{S}_0$ be the collection of all trace-free $2\times 2$ symmetric matrices, which is two-dimensional. We can define an inner product
    \begin{equation}
        \inn{A}{B}:= \mathrm{tr}(AB). 
    \end{equation}
\end{claim}

\begin{claim}\label{260509claim8_2}
    For $A, B\in \mathcal{S}_0$, we have 
    \begin{equation}
        A B + B A = \operatorname{tr}(AB) I. 
    \end{equation}
\end{claim}
The proofs for these two claims are elementary, and we leave them out.  By the Cayley-Hamilton theorem, we obtain 
\begin{equation}\label{260509e8_4}
    (R_-(s))^2= - (\det R_-(s)) I= \mu^2 I. 
\end{equation}
We take the derivative in $s$ on both sides of \eqref{260509e8_4} and obtain 
\begin{equation}\label{260509e8_8}
    R_-(s) R'_-(s)+ R'_-(s) R_-(s)=0.
\end{equation}
Note that $R_-(s)$ is trace-free, and therefore $R'_-(s)$ is trace-free as well. We combine \eqref{260509e8_8}, Claim \ref{260509claim8_1} and Claim \ref{260509claim8_2}, and obtain that $R'_-(s)$ is orthogonal to $R_-(s)$ in $\mathcal{S}_0$. However, previously we also obtained that $R'_-(0)$ is a scalar multiple of $R_-(0)$. All these imply that 
\begin{equation}
    R'_-(0)=0.
\end{equation}
We can continue this process by differentiating more on both sides of \eqref{260509e8_4}, and eventually obtain 
\begin{equation}
    R^{(m)}_-(0)=0, \ \forall m\ge 1.
\end{equation}
This finishes the proof of the entire corollary.

\section{Proof of Theorem \ref{260509theorem2_1}: Item (3)}

Item (3) consists of two parts. The first part is to show that the matrix \eqref{260509e2_2} has rank four for all small enough $\epsilon$ if and only if $R(s)$  is not simultaneously diagonalizable in $s$, and the second part is the equivalence of items (3.1)-(3.5).  We will only prove the second part, and the first part follows immediately. \\

Let us be more precise. Assume that $R(s)$  is not simultaneously diagonalizable in $s$. We can therefore find $k$ such that $R(0), \dots, R^{(k-3)}(0)$ cannot be simultaneously diagonalized, that is, item (3.3) holds. By that (3.3) and (3.1) are equivalent, we can conclude that the matrix \eqref{260509e2_2} has rank four for all small enough $\epsilon$. 

Now we assume that the matrix \eqref{260509e2_2} has rank four for all small enough $\epsilon$. We can conclude immediately that there exist $k$ and $\epsilon>0$ such that the first $k$ columns of the matrix \eqref{260509e2_2} has rank four. By that (3.2) and (3.3) are equivalent, we obtain that $R(0), \dots, R^{(k-3)}(0)$ cannot be simultaneously diagonalized, and therefore $R(s)$  is not simultaneously diagonalizable in $s$. \\

In the rest of this section, we will prove the equivalence of (3.1)-(3.5). That (3.1) and (3.2) are equivalent follows immediately from analyticity. In Subsection \ref{260505sub4_1} we will prove that (3.2) implies (3.3), in Subsection \ref{260603subsection4_3} we will prove that (3.3) implies (3.2), and in the last subsection we will prove that (3.3), (3.4) and (3.5) are mutually equivalent.

\subsection{Preliminary calculations}\label{260601subsection5_1}

Let $\gamma$ be a geodesic with unit speed. We make the same normalization as in \eqref{260502e1_5} that 
\begin{equation}\label{260504e3_7}
\gamma(0)=(0, 0, 0), \ \ \gamma(s)=(0, 0, s), \forall s.
\end{equation}
We take the phase function 
\begin{equation}\label{260504e3_8}
\phi_{\epsilon}(x, t; \xi):= \dist((x, t), (\xi, \epsilon)),
\end{equation}
and compute the $4\times \infty$ matrix \eqref{260309e1_4}  at 
\begin{equation}\label{260504e3_9}
x_0=(0, 0), \ \ \xi=(0, 0).
\end{equation}
Characteristic curves for the distance function are precisely given by geodesics (see for instance Lemma 3.2 in \cite{DGGZ24}), and therefore 
\begin{equation}
X(t; x_0; \xi)=0, \ \ \forall t,
\end{equation}
due to the normalization in \eqref{260504e3_7}, where $X$ was introduced in \eqref{260502e1_1}.

We will rely heavily on several calculations from \cite{DGGZ24}. In order for our notation to be consistent with the ones in \cite{DGGZ24}, we will switch from $\xi$ to $y$.  We follow (3.32)--(3.34) in \cite{DGGZ24}, and denote 
\begin{equation}\label{260504e3_12}
\Phi(x, t; y, \tau):= \dist((x, t), (y, \tau)),
\end{equation}
\begin{equation}
\phi_{0, \epsilon}(x, t ; y):=\phi_\epsilon(x, t ; y)-\phi_\epsilon(0,0 ; y),
\end{equation}
and 
\begin{equation}\label{260504e3_14}
\Phi_0(x, t ; y, \tau):=\Phi(x, t ; y, \tau)-\Phi(0,0 ; y, \tau) .
\end{equation}
Note that in \eqref{260504e3_12} and \eqref{260504e3_14}, we are no longer fixing $\tau=\epsilon$ but treating $\tau$ as  a variable; this is necessary as we will be computing covariant derivatives of the distance function in all its variables (not just $y$).

Recall from \eqref{260504e1_10} that we need to compute
\begin{equation}
\partial_{y_i} \partial_{y_j} \phi_{0, \epsilon}.
\end{equation}
By \cite[Claim 3.4]{DGGZ24}, we have 
\begin{equation}\label{260601e5_9zz}
\left.\frac{\partial}{\partial y_i} \frac{\partial}{\partial y_j} \phi_{0, \epsilon}\right|_{\substack{(x, t)=\gamma(s)\\ y=0}}=
\left.\left(\operatorname{Hessian} \Phi_0\right)\right|_{\substack{(x, t)=\gamma(s)\\(y, \tau)=(0, \epsilon)}}\left(\frac{\partial}{\partial y_i}, \frac{\partial}{\partial y_j}\right)
\end{equation}
for every $s\in [0, \epsilon)$ and $1\le i, j\le 2$, where $\mathrm{Hessian}$ refers to the covariant Hessian in the $(y, \tau)$ variables. \\

Let us record here an equation that is slightly stronger than \eqref{260601e5_9zz}.

\begin{claim}\label{260601claim5_1}
We have 
\begin{equation}\label{260601e5_32zz}
    \left.\frac{\partial}{\partial y_i} \frac{\partial}{\partial y_j} \phi_{\epsilon}\right|_{\substack{(x, t)=\gamma(s)\\ y=0}}= \left.\left(\operatorname{Hessian} \Phi\right)\right|_{\substack{(x, t)=\gamma(s)\\(y, \tau)=(0, \epsilon)}}\left(\frac{\partial}{\partial y_i}, \frac{\partial}{\partial y_j}\right),
\end{equation}
for all $1\le i, j\le 2$. 
\end{claim}
\begin{proof}[Proof of Claim \ref{260601claim5_1}]
    The proof of \eqref{260601e5_32zz} is in some sense similar to that of \eqref{260601e5_9zz} but with a subtle difference. Let us write down a proof here. By definition, \footnote{Here we add a bar on top of $\nabla$ just to distinguish $\bar{\nabla}$ from $\nabla$: Under this notation, $\bar{\nabla} f$ is a vector field, and $\nabla f$ is a one-form, where $f$ is a function. }
\begin{equation}
\begin{aligned}
\left(\operatorname{Hessian} \Phi\right)\left(\frac{\partial}{\partial y_i}, \frac{\partial}{\partial y_j}\right) & =g\left(\nabla_{\frac{\partial}{\partial y_i}}\left(\bar{\nabla} \Phi\right), \frac{\partial}{\partial y_j}\right) \\
& =\nabla_{\frac{\partial}{\partial y_i}}\left(\nabla_{\frac{\partial}{\partial y_j}} \Phi\right)-g\left(\bar{\nabla} \Phi, \nabla_{\frac{\partial}{\partial y_i}} \frac{\partial}{\partial y_j}\right) .
\end{aligned}
\end{equation}
Note that 
\begin{equation}
    \nabla_{
    \frac{\partial}{\partial y_i} 
    }
    \frac{\partial}{\partial y_j}=
    \sum_{k=1}^3
    \Gamma^k_{ij} \frac{\partial}{\partial y_k},
\end{equation}
with $y=(y_1, y_2)$ and $\partial/\partial y_3$ refers to the direction $\dot{\gamma}$.  Moreover, note that in Fermi coordinates, the Christoffel symbols all vanish when we are on the geodesic $\gamma$, see for instance \cite[page 5]{Sog99} or (3.22) in \cite{DGGZ24}. This finishes the proof of \eqref{260601e5_32zz}.
\end{proof}

Back to \eqref{260601e5_9zz}. We follow (3.63)-(3.66) in \cite{DGGZ24} and denote 
\begin{equation}
W_{ij}(s, \epsilon):=
\left.\left(\operatorname{Hessian} \Phi_0\right)\right|_{\substack{(x, t)=\gamma(s)\\(y, \tau)=(0, \epsilon)}}\left(\frac{\partial}{\partial y_i}, \frac{\partial}{\partial y_j}\right),
\end{equation}
\begin{equation}
W(s, \epsilon):=\left[\begin{array}{ll}
W_{11}(s, \epsilon) & W_{12}(s, \epsilon) \\
W_{21}(s, \epsilon) & W_{22}(s, \epsilon)
\end{array}\right],
\end{equation}
and 
\begin{equation}
\mathcal{W}(s, \epsilon):= \det W(s, \epsilon).
\end{equation}
We will need to compute the rank of the matrix 
\begin{equation}\label{260504e3_20}
\left[\begin{array}{ccccc}
\partial_s \mathcal{W}(0, \epsilon) & \partial_s^2 \mathcal{W}(0, \epsilon) & \dots & \partial_s^k \mathcal{W}(0, \epsilon) & \dots \\
\partial_s W_{11}(0, \epsilon) & \partial_s^2 W_{11}(0, \epsilon) & \dots & \partial_s^k W_{11}(0, \epsilon) & \dots \\
\partial_s W_{12}(0, \epsilon) & \partial_s^2 W_{12}(0, \epsilon) & \dots & \partial_s^k W_{12}(0, \epsilon) & \dots \\
\partial_s W_{22}(0, \epsilon) & \partial_s^2 W_{22}(0, \epsilon) & \dots & \partial_s^k W_{22}(0, \epsilon) & \dots
\end{array}\right]
\end{equation}
 for every $\epsilon$. We repeat (3.67)--(3.88) in \cite{DGGZ24}, and obtain that 
\begin{equation}\label{260505e3_21}
W_{ij}(s, \epsilon)= 
\partial_{s'} a_{ij}(s, \epsilon)-\partial_{s'} a_{ij}(0, \epsilon),
\end{equation}
where 
\begin{equation}
A(s, s')=\begin{bmatrix}
a_{11}(s, s') & a_{12}(s, s')\\
a_{21}(s, s') & a_{22}(s, s')
\end{bmatrix}
\end{equation}
solves the system of ODEs
\begin{equation}\label{260504e3_23}
\begin{aligned}
& \partial_{s^{\prime}}^2 A\left(s, s^{\prime}\right)+A\left(s, s^{\prime}\right) R\left(s^{\prime}\right)=0 \\
& A(s, s)=0, A(s, \epsilon)=I_{2 \times 2}
\end{aligned}
\end{equation}
with $R(s')$ given by \eqref{260504e1_16} and \eqref{260504e1_17}.

\subsection{Item (3.2) implies item (3.3)}\label{260505sub4_1}

We will argue by contradiction, and assume that $R(0), R'(0), \dots, R^{k-3}(0)$ can be simultaneously diagonalized. Our goal is to show that for every $\epsilon>0$, the first $k$ columns of the matrix \eqref{260509e2_2} has rank $\le 3$.\\

We keep the normalization in \eqref{260504e3_7}, and recycle the notation in \eqref{260504e3_8}-- \eqref{260504e3_20}. To show the matrix \eqref{260504e3_20} has rank $\le 3$, we will need to solve the ODE system \eqref{260504e3_23}. 

Let us start by recalling some more\footnote{Below the statement of Theorem \ref{260509theorem2_1} we already recalled some of these calculations. } preliminary calculations in \cite{DGGZ24}, and we will also take this opportunity to raise a question about a potentially more quantitative version of item (3.1) in the statement of Theorem \ref{260509theorem2_1}.

In \cite{DGGZ24}, in order to solve the ODE system \eqref{260504e3_23}, two auxiliary ODE systems were introduced. They are 
\begin{equation}
\left\{
\begin{aligned}
&\partial_t^2 B_1(t)+B_1(t)R(t)=0,\\
&B_1(0)=I,\qquad \partial_t B(0)=0, 
\end{aligned}
\right.
\end{equation}
and 
\begin{equation}
\left\{
\begin{aligned}
&\partial_t^2 B_2(t)+B_2(t)R(t)=0,\\
&B_2(0)=0,\qquad \partial_t B(0)=I,
\end{aligned}
\right.
\end{equation}
see (3.113) and (3.114) in \cite{DGGZ24}.
 We can then write 
\begin{equation}
A(s, t)= C_1(s) B_1(t)+ C_2(s) B_2(t),
\end{equation}
for some scalar matrices $C_1(s), C_2(s)$, and we can use the boundary conditions on $A(s, t)$ to determine the derivatives of $C_1, C_2$ at $s=0$. 
\begin{lemma}[\cite{DGGZ24}]\label{260505lemma4_1}
Under the above notation, we have 
\begin{align*}
\partial_sW(0,\epsilon)&=-B_2^{-1}(\epsilon)B'_1(\epsilon)+
B_2^{-1}(\epsilon)B_1(\epsilon)B_2^{-1}(\epsilon) B'_2(\epsilon),\\
\partial_s^2W(0,\epsilon)&=2 B_2^{-1}(\epsilon) B_1(\epsilon) \partial_sW(0,\epsilon),\\
\partial_s^3W(0,\epsilon)&=-\left(6 B_2^{-1}(\epsilon) B_1(\epsilon) B_2^{-1}(\epsilon) B_1(\epsilon) B_2^{-1}(\epsilon)+2 B_2^{-1}(\epsilon) R(0)\right)\\
& \quad \times \left(B_1^{\prime}(\epsilon)-B_1(\epsilon) B_2^{-1}(\epsilon) B_2^{\prime}(\epsilon)\right),
\end{align*}
and a similar expression for $\partial^4_s W(0, \epsilon)$. 
\end{lemma}

For the three equations in Lemma \ref{260505lemma4_1}, we refer to (3.123), (3.125), (3.129) in \cite{DGGZ24}. After obtaining Lemma \ref{260505lemma4_1}, the authors of \cite{DGGZ24} expanded $\partial^{k'}_s W(0, \epsilon)$, which is equal to 
\begin{equation}
\partial^{k'}_s \partial_{s'}
A(0, \epsilon)
\end{equation}
by \eqref{260505e3_21}, 
in Taylor's series in the $\epsilon$ variable, for all $k'=1, 2, 3, 4$. 
\begin{lemma}[Claim 3.8 in \cite{DGGZ24}]
Under the above notation, we have 
\begin{equation}
\begin{aligned}
\partial_s \partial_{s^{\prime}} A(0, \epsilon) & =\frac{1}{\epsilon^2}\left(I_{2 \times 2}+\frac{R_0}{3} \epsilon^2+\frac{R_0^{\prime}}{6} \epsilon^3+O\left(\epsilon^4\right)\right), \\
\partial_s^2 \partial_{s^{\prime}} A(0, \epsilon) & =\frac{2!}{\epsilon^3}\left(I_{2 \times 2}+\frac{R_0^{\prime}}{12} \epsilon^3+O\left(\epsilon^4\right)\right), \\
\partial_s^3 \partial_{s^{\prime}} A(0, \epsilon) & =\frac{3!}{\epsilon^4}\left(I_{2 \times 2}+O\left(\epsilon^4\right)\right), \\
\partial_s^4 \partial_{s^{\prime}} A(0, \epsilon) & =\frac{4!}{\epsilon^5}\left(I_{2 \times 2}+O\left(\epsilon^4\right)\right) .
\end{aligned}
\end{equation}
\end{lemma}
In the end, \cite{DGGZ24} computes the determinant of the leading $4\times 4$ matrix of \eqref{260504e3_20} directly, and obtain that it is given by 
\begin{equation}\label{260505e4_6}
 -\frac{1}{144}\pnorm{
 (\mf{g}_{33, 11}-\mf{g}_{33, 22}) \mf{g}_{33, 123}
 - (\mf{g}_{33, 113}-\mf{g}_{33, 223}) \mf{g}_{33, 12}
 } \epsilon^5+ O(\epsilon^6), 
\end{equation}
see \eqref{260505e1_25a} and \eqref{260505e1_26a} for the new notation above. \\

Note that if $R(0), R'(0)$ do not commute, that is item (3.3) of Theorem \ref{260509theorem2_1} holds with $k=4$, then the coefficient of $\epsilon^5$ in \eqref{260505e4_6} does not vanish, and therefore the first four columns of the matrix \eqref{260504e3_20} has rank four, whenever $\epsilon$ is chosen to be small enough. This is exactly how \cite{DGGZ24} proves that when $k=4$, item (3.3) of Theorem \ref{260509theorem2_1} implies item (3.1).

With the help of the calculation \eqref{260505e4_6} and with the assumption $k=4$, we actually know that whenever item (3.1) of Theorem \ref{260509theorem2_1} holds, it always holds very ``quantitatively", that is, as $\epsilon\to 0$, the determinant cannot tend to zero faster than $\epsilon^5$. 

When trying to prove Theorem \ref{260509theorem2_1} for general $k\ge 4$, 
we tried to generalize the above approach, that is, the approach of Taylor's expansions in $\epsilon$, but failed dramatically. We even failed dramatically at $k=4$ when computing the coefficient of $\epsilon^{j}$ in \eqref{260505e4_6} for large $j\in \N$. 

As a consequence, when we are in the case that item (3.1) of Theorem \ref{260509theorem2_1} holds, we only know that there exists one $4\times 4$ minor coming from the first $k$ columns whose determinant is not identically zero in $\epsilon$, but we do not know the non-vanishing order of this determinant (other than the case $k=4$ which was calculated in \eqref{260505e4_6}). \\

Before we start explaining our new approach, let us record certain simplifications of the expressions in Lemma \ref{260505lemma4_1}. The idea behind this simplifying process will also be used in our new approach. Moreover, we will see that the simplified expressions will help us understand the recursion relations for general $k$, see for instance Proposition \ref{260505prop4_5}. 

\begin{lemma}\label{prop:Wrons}
To simplify the notations, we write 
$R_0=R(0), R'_0= R'(0)$ and $B_1=B_1(\epsilon)$, $B_2=B_2(\epsilon)$. We have
\begin{align*}
\partial_sW(0,\epsilon)&=B_2^{-1}(B_2^T)^{-1},\\
\partial_s^2W(0,\epsilon)&=B_2^{-1} (2! B_1 B_2^{-1})(B_2^T)^{-1},\\
\partial_s^3W(0,\epsilon)&= B_2^{-1}(3!(B_1 B_2^{-1})^2+2R_0)(B_2^T)^{-1},\\
\partial_s^4W(0,\epsilon)&=B_2^{-1}(4! (B_1 B_2^{-1})^3+8(B_1 B_2^{-1}R_0+R_0B_1 B_2^{-1})+2R_0')(B_2^T)^{-1}.
\end{align*}
\end{lemma}
\begin{proof}[Proof of Lemma \ref{prop:Wrons}]
The proof is via a Wronskian argument: If $B,\bar{B}$ are solutions of
\begin{equation}
B''(t)+B(t)R(t)=0,
\end{equation}
 then
\begin{equation}\label{260505e4_8zz}
B(\bar{B}^T)'-B'\bar{B}^T \text{ is constant.}
\end{equation}
The proof is straightforward and is left out, and such ideas were already used in  \cite{KV86}.\\

We apply \eqref{260505e4_8zz} to $B_1,B_2$, and obtain 
\begin{align}
B_1(B_2^T)'-B_1'B_2^T&=I,\label{eq:Wrons B1'}\\
B_2(B_2^T)'-B_2'B_2^T&=0. \label{eq:Wrons B2'}
\end{align}
For the first equation in the lemma: Plugging \eqref{eq:Wrons B1'} and \eqref{eq:Wrons B2'} in the first equation in Lemma \ref{260505lemma4_1}, we have
\begin{multline}
\partial_sW=-B_2^{-1}(B_1(B_2^T)'(B_2^T)^{-1}-(B_2^T)^{-1})\\
+B_2^{-1}B_1B_2^{-1} (B_2(B_2^T)'(B_2^T)^{-1})
=B_2^{-1}(B_2^T)^{-1}.
\end{multline}
The second equation follows from the first equation of the current lemma and the second equation in Lemma \ref{260505lemma4_1}.

For the third equation: Plugging \eqref{eq:Wrons B1'} and \eqref{eq:Wrons B2'}, we can compute the second factor in the third equation in Lemma \ref{260505lemma4_1}, and obtain 
\begin{multline}
B_1^{\prime}-B_1 B_2^{-1} B_2^{\prime}=B_1(B_2^T)'(B_2^T)^{-1}-(B_2^T)^{-1}\\
-B_1 B_2^{-1} B_2(B_2^T)'(B_2^T)^{-1}=-(B_2^T)^{-1}.
\end{multline}
Combining with the third equation in Lemma \ref{260505lemma4_1}, we can obtain the conclusion.

The last equation can be proven similarly, and we leave out the details. 
\end{proof}

Let us begin explaining our new approach. 
Consider the initial value problem:
\begin{equation}\label{260504e4_1a}
\left\{
\begin{aligned}
&\partial_t^2 B(s,t)+B(s,t)R(t)=0,\\
&B(s,s)=0,\qquad \partial_t B(s,s)=I. 
\end{aligned}
\right.
\end{equation}
\begin{lemma}\label{260601lemma5_4}
    Let $\sec(s, \epsilon)$ be the second fundamental form at $\gamma(\epsilon)$ of the geodesic sphere centered at $\gamma(s)$ (and of radius $\epsilon-s$). Then 
    \begin{equation}\label{260504e4_1}
\sec(s,\epsilon)=B^{-1}(s,\epsilon)\partial_t B(s,\epsilon)
\end{equation}
\end{lemma}
\begin{proof}[Proof of Lemma \ref{260601lemma5_4}]
Recall the definition of $\Phi$ in \eqref{260504e3_12}. By the definition of the second fundamental form, we first have that  
\begin{equation}\label{260601e5_31zz}
    \left.\left(\operatorname{Hessian} \Phi\right)\right|_{\substack{(x, t)=\gamma(s)\\(y, \tau)=(0, \epsilon)}}\left(\frac{\partial}{\partial y_i}, \frac{\partial}{\partial y_j}\right)= \sec_{ij}(s, \epsilon),
\end{equation}
for all $1\le i, j\le 2$. Recall the definition of the matrix $A(s, s')$ in \eqref{260504e3_23}. By repeating (3.67)-(3.87) in \cite{DGGZ24}, we obtain that 
\begin{equation}
    \left.\left(\operatorname{Hessian} \Phi\right)\right|_{\substack{(x, t)=\gamma(s)\\(y, \tau)=(0, \epsilon)}}\left(\frac{\partial}{\partial y_i}, \frac{\partial}{\partial y_j}\right)=
    \partial_{s'} a_{ij}(s, \epsilon). 
\end{equation}
In the end, note that 
    \begin{equation}
        A(s, s')= B^{-1}(s, \epsilon) B(s, s').
    \end{equation}
This finishes the proof of lemma. 
\end{proof}

Note that 
\begin{equation}\label{260602e4_37}
W(s,\epsilon)=\sec(s,\epsilon)-\sec(0,\epsilon).
\end{equation}
To study the second fundamental form $\sec(s, \epsilon)$, we introduce the following two crucial ODE systems: 
\begin{equation}
\left\{
\begin{aligned}
&A_1''(t)+A_1(t)R(t)=0,\\
&A_1(\epsilon)=I,\qquad A_1'(\epsilon)=0,\\
\end{aligned}
\right.
\end{equation}
and 
\begin{equation}\label{260505e4_17tt}
\left\{
\begin{aligned}
&A_2''(t)+A_2(t)R(t)=0,\\
&A_2(\epsilon)=0,\qquad A_2'(\epsilon)=-I.
\end{aligned}
\right.
\end{equation}
\begin{claim}\label{260504claim4_1}
Let $\sec(\epsilon, s)$ be the second fundamental form at $\gamma(s)$ of the geodesic sphere centered at $\gamma(\epsilon)$ (and of radius $\epsilon-s$). Then, we have the reversed second fundamental form,
\begin{equation}\label{eq: Reverse sec}
\sec(\epsilon,s)=-A_2^{-1}(s)A_2'(s),
\end{equation}
and the reversed Riccati equation:
\begin{equation}\label{eq: Reverse Riccati}
\partial_s\sec(\epsilon,s)=\sec(\epsilon,s)^2+R(s).
\end{equation}
\end{claim}
\begin{proof}[Proof of Claim \ref{260504claim4_1}]
The reversed Riccati equation follows from \eqref{eq: Reverse sec} and \eqref{260505e4_17tt}. Now, we will prove \eqref{eq: Reverse sec}. Along the geodesic $\gamma(t)$, we introduce a reversed time parameter $\tilde{t}=-t$ and define the reversed curvature matrix:
$$
\tilde{R}(\tilde{t}):=R(-\tilde{t}).
$$
With the $\tilde{t}$ parameter, since $-\epsilon<-s$, we have that  
\begin{equation}\label{260601e5_31zzr}
\tilde{\sec}_{ij}(-\epsilon,-s)=\left.\left(\operatorname{Hessian} \Phi\right)\right|_{\substack{(x, \tilde{t})=(0,-\epsilon)\\(y, \tilde{\tau})=(0,-s)}}\left(\frac{\partial}{\partial y_i}, \frac{\partial}{\partial y_j}\right)
\end{equation}
for all $1\le i, j\le 2$. Denote the reversed second fundamental form
$$
\sec(\epsilon,s)=\tilde{\sec}(-\epsilon,-s).
$$
Then, by the same proof of \eqref{260504e4_1}, we have
\begin{equation}\label{eq:reverse sec proof}
\left.\left(\operatorname{Hessian} \Phi\right)\right|_{\substack{(x, \tilde{t})=(0,-\epsilon)\\(y, \tilde{\tau})=(0,-s)}}=
\tilde{B}^{-1}(-\epsilon,-s)\partial_{\tilde{t}}\tilde{B}(-\epsilon,-s),
\end{equation}
where $\tilde{B}(-\epsilon,\tilde{t})$ solves the reversed Jacobi equation:
\begin{equation}\label{260504e4_1ar}
\left\{
\begin{aligned}
&\partial_{\tilde{t}}^2 \tilde{B}(-\epsilon,\tilde{t})+\tilde{B}(-\epsilon,\tilde{t})\tilde{R}(\tilde{t})=0,\\
&\tilde{B}(-\epsilon,-\epsilon)=0,\qquad \partial_{\tilde{t}}\tilde{B}(-\epsilon,-\epsilon)=I. 
\end{aligned}
\right.
\end{equation}
Changing the parameter, we can see that $A_2(t)$ and $\tilde{B}(-\epsilon,-t)$ satisfy the same ODE and initial values. By the uniqueness of the solution, we obtain
$$
A_2(t)=\tilde{B}(-\epsilon,-t).
$$
Combining with \eqref{260601e5_31zzr} and \eqref{eq:reverse sec proof}, we can obtain \eqref{eq: Reverse sec}. 
\end{proof}

We find an important relation between $\sec(s,\epsilon)$ and $\sec(\epsilon,s)$:
\begin{proposition}\label{260504prop4_1}
We have
\begin{equation}\label{eq:second derivative}
\partial_s^2\sec(s,\epsilon)=2(A_2^{T})^{-1}(s)\sec(\epsilon,s)A_2^{-1}(s).
\end{equation}
More generally, for $k\geq 1$, if we define
\begin{equation}\label{eq:def of I_k}
I_k(\epsilon,s)=A_2^T(s)\partial_s^kW(s,\epsilon)A_2(s),
\end{equation}
then we have $I_1=I$ and a recursive relation
\begin{equation}\label{eq:I_k recursion}
I_{k+1}(\epsilon,s)=\partial_sI_k(\epsilon,s)+\sec(\epsilon,s)I_k(\epsilon,s)+I_k(\epsilon,s)\sec(\epsilon,s).
\end{equation}
\end{proposition}

One first difficulty in computing the $4\times \infty$ matrix \eqref{260504e3_20} is how to calculate partial derivatives in the $s$ variable. More precisely, from the defining equation \eqref{260504e4_1a}, we already know how to take successive derivatives in the $t$ variable, which is by taking derivatives in $t$ on both sides of the first equation in \eqref{260504e4_1a}. However, finding a convenient way of taking derivatives in the $s$ variable is more challenging. 

The dual relation \eqref{eq:second derivative} in Proposition \ref{260504prop4_1} and the reversed Riccati equation \eqref{eq: Reverse Riccati} tell us exactly how to take all derivatives in the $s$ variable. 

This relation between $\sec(s, \epsilon)$ and $\sec(\epsilon, s)$ seems to be new in the literature, although equations in the same spirit have been found in \cite[Proposition 6]{KV86}:
\begin{align}
\sec(s,\epsilon)&=A_1(s)A_2^{-1}(s),\label{eq:zero derivative}\\
\partial_s\sec(s,\epsilon)&=(A_2^T)^{-1}(s)A_2^{-1}(s),\label{eq:first derivative}
\end{align}
that is, the authors in \cite{KV86} only need to take derivatives up to order one in their problem. However, in the current problem, 
we cannot reveal the relation between $\sec(s,\epsilon)$ and $\sec(\epsilon,s)$ without taking the second derivative. Moreover, we also obtain the recursion relation \eqref{eq:I_k recursion} of all orders. To be self-contained, we will provide the proofs of \eqref{eq:zero derivative} and \eqref{eq:first derivative}.

\begin{proof}[Proof of Proposition \ref{260504prop4_1}]
Note that $A_1(t)$ and $B(s, t)$ satisfy the same differential relations in the $t$ variable (with different initial data), and therefore the Wronskian 
\begin{equation}
A_1(t)\partial_tB^T(s,t)-A_1'(t)B^T(s,t)
\end{equation}
is constant in the $t$ variable. We 
evaluate the Wronskian 
at $t=\epsilon$ and $t=s$, and obtain
\begin{equation}\label{eq:A1(s)}
A_1(s)=\partial_t B^T(s,\epsilon).
\end{equation}
Similarly, we evaluate the Wronskian 
$$
B(s,t)(A_2^T(t))'-\partial_tB(s,t)A_2^T(t)
$$
at $t=\epsilon$ and $t=s$, we can obtain
\begin{equation}\label{eq:A2(s)}
A_2^T(s)=B(s,\epsilon).
\end{equation}
Then, we can prove equation \eqref{eq:zero derivative}:
\begin{equation}\label{260504e4_14}
\sec(s,\epsilon)=B^{-1}(s,\epsilon)\partial_t B(s,\epsilon)=(A_2^T(s))^{-1}A_1(s)^T=A_1(s)A_2^{-1}(s).
\end{equation}
Moreover, we apply the same Wronskian argument as above and obtain 
\begin{align}
A_1(s)(A_2(s)^T)'-A_1'(s)A_2(s)^T&=-I,\label{eq:Wrons A1'}\\
A_2(s)(A_2(s)^T)'-A_2'(s)A_2(s)^T&=0,
\label{eq:Wrons A2'}
\end{align}
for all $s$. 
Then, taking $\partial_s$ of equation \eqref{260504e4_14} and adding the two equations above, we can prove equation \eqref{eq:first derivative}:
\begin{equation}\label{260504e4_17}
\partial_s\sec(s,\epsilon)=A_1'A_2^{-1}-A_1A_2^{-1}A_2'A_2^{-1}=(A_2^T)^{-1}A_2^{-1}.
\end{equation}
Finally, taking $\partial_s$ of equation \eqref{260504e4_17} and adding equations \eqref{eq:Wrons A2'}, \eqref{eq: Reverse sec}, we have
\begin{align*}
\partial_s^2\sec(s,\epsilon)&=-(A_2^T)^{-1}(A_2^T)'(A_2^T)^{-1}A_2^{-1}-(A_2^T)^{-1}A_2^{-1}A_2'A_2^{-1}\\
&=-2(A_2^T)^{-1}A_2^{-1}A_2'A_2^{-1}\\
&=2(A_2^T)^{-1}(s)\sec(\epsilon,s)A_2^{-1}(s).
\end{align*}
This finishes the proof of \eqref{eq:second derivative}. \\

We turn to the proof of the more general formula \eqref{eq:I_k recursion}. Note that when $k=1$, the equation \eqref{eq:I_k recursion} becomes 
\begin{equation}
I_1(\epsilon, s)= A_2^T(s) \partial_s W(s, \epsilon) A_2(s)=
 A_2^T(s) \partial_s \sec(s, \epsilon) A_2(s),
\end{equation}
which, combined with \eqref{eq:first derivative}, yields $I_1=I$. For a general $k$, let us write 
\begin{equation}
I_{k+1}(\epsilon, s)=
A_2^T(s) \partial_s \pnorm{
\partial^k_s W(s, \epsilon) 
} A_2(s).
\end{equation}
Note that 
\begin{equation}
\partial^k_s W=
(A_2^T)^{-1} I_k A_2^{-1},
\end{equation}
and therefore 
\begin{equation}
I_{k+1}= A_2^T \partial_s \pnorm{
(A_2^T)^{-1} I_k A_2^{-1}
}
A_2.
\end{equation}
By the chain rule, we obtain 
\begin{equation}
I_{k+1}=
-A_2^T (A_2^T)^{-1} (A_2^T)'  (A_2^T)^{-1} I_k A_2^{-1} A_2+ \partial_s I_k- \partial_s I_k A_2^{-1} A'_2 A_2^{-1} A_2,
\end{equation}
which is further equal to 
\begin{equation}\label{260505e4_25}
-(A_2^T)'  (A_2^T)^{-1} I_k A_2^{-1} A_2+ \partial_s I_k- \partial_s I_k A_2^{-1} A'_2.
\end{equation}
By \eqref{eq: Reverse sec} and by the fact that the second fundamental form $\sec(\epsilon, s)$ is symmetric, we obtain 
\begin{equation}
A_2^{-1}(s)A'_2(s)=- \sec(\epsilon, s)=
(A'_2(s))^T (A_2^{-1}(s))^T.
\end{equation}
This, combined with \eqref{260505e4_25}, implies the desired recursion relation \eqref{eq:I_k recursion}. This finishes the proof of the proposition. 
\end{proof}
\begin{remark}\label{260504remark4_3}
Let us explain the relation between Proposition \ref{260504prop4_1} and Lemma \ref{prop:Wrons}.  First of all, we take $s=0$ in \eqref{eq:A2(s)}, and obtain 
\begin{equation}
A_2^T(0)=B(0,\epsilon)=B_2(\epsilon).
\end{equation}
Combining with the definition of $I_k$, there is
\begin{equation}
\partial_s^kW(0,\epsilon)=B_2^{-1}(\epsilon)I_k(\epsilon,0)(B_2^T)^{-1}(\epsilon).
\end{equation}
From this and Lemma \ref{prop:Wrons}, we can see explicitly how $I_k(\epsilon, 0)$ looks like for $k=1, 2, 3, 4$. 

Next, let us try to explain the matrix $B_1 B_2^{-1}$ that appears in Lemma \ref{prop:Wrons}. Indeed, we claim that 
\begin{equation}
B_1(\epsilon)B_2(\epsilon)^{-1}=\sec(\epsilon,0),
\end{equation}
which is precisely the matrix $\sec$ that appears in the recursive relation \eqref{eq:I_k recursion}.

To see this, by \eqref{eq: Reverse sec}, it suffices to prove 
\begin{equation}\label{260505e4_42tt}
B_1(\epsilon)B_2(\epsilon)^{-1}= - A_2^{-1}(0) A'_2(0),
\end{equation}
where $A_2$ is defined via \eqref{260505e4_17tt}. To prove \eqref{260505e4_42tt}, we see that via linearity we can write 
\begin{equation}
A_2(t) = C_1 B_1(t) + C_2 B_2(t)\end{equation}
for two scalar matrices $C_1, C_2$. By the boundary condition at $t=0$, we can obtain 
\begin{equation}
A_2(t) = A_2(0)B_1(t) + A_2'(0)B_2(t). 
\end{equation}
In the end, we apply the boundary condition at $t=\epsilon$, and will obtain \eqref{260505e4_42tt}. 
\end{remark}

Next, we will do a decomposition for $I_k(\epsilon,s)$ and find a recursive relation.

\begin{proposition}\label{260505prop4_5}
We can write
\begin{equation}\label{eq:decomposition}
I_k(\epsilon,s)=a_k(s)\sec(\epsilon,s)+b_k(s)I+\mathcal{P}_k(R(s),\ldots,R^{(k-3)}(s))
\end{equation}
where $\mathcal{P}_k$ is a polynomial whose coefficients depend on $s$. To be more precise,\footnote{Regarding  the formula \eqref{260505e4_28}, products of $R^{(i)}$ start to appear when $k=5$, and we do not see these product terms in Proposition \ref{prop:Wrons}. }
\begin{equation}\label{260505e4_28}
\mc{P}_k=\sum_{i_1+\ldots+i_l\leq k-3}c_{i_1,\ldots,i_l}(s)R^{(i_1)}(s)R^{(i_2)}(s)\cdots R^{(i_l)}(s).
\end{equation}
Moreover we have $a_1=0$ and a recursive relation
\begin{equation}\label{eq:coeffi recursion}
a_{k+1}=a_k'+2a_k\mathrm{tr}\sec(\epsilon, s) +\mathrm{tr}I_k(\epsilon, s).
\end{equation}
Here $a_k, b_k$ and $c_{i_1, \dots, i_{l}}$ all depend on $\epsilon$ as well, but we ignore the dependence for the sake of simplifying notations. 
\end{proposition}
\begin{proof}[Proof of Proposition \ref{260505prop4_5}]
We will prove \eqref{eq:decomposition} by induction. When $k=1$, $I_1=I$. The decomposition holds. Next, we will prove it for $I_{k+1}$. Combining with \eqref{eq:I_k recursion}, we have
$$
I_{k+1}=\partial_s(a_k\sec+b_kI+\mc{P}_k)+(I_k\sec+\sec I_k).
$$
For the first term, we use \eqref{eq: Reverse Riccati} and obtain 
\begin{equation}\label{260505e4_30}
\partial_s(a_k\sec+b_kI+\mc{P}_k)=a_k'\sec+a_k(\sec^2+R)+b_k'I+\partial_s \mc{P}_k.
\end{equation}
To handle the term $\sec^2$, we apply the Cayley-Hamilton theorem, which says that 
\begin{equation}
\sec^2=(\mathrm{tr} \sec) \sec-(\det \sec) I,
\end{equation}
and obtain 
\begin{align*}
\eqref{260505e4_30}&=a_k'\sec+a_k((\mathrm{tr}\sec) \sec-(\det\sec) I+R)+b_k'I+\partial_s \mc{P}_k\\
&=(a_k'+a_k\mathrm{tr}\sec) \sec +(b_k'-a_k\det\sec)I+(\partial_s \mc{P}_k+a_kR).
\end{align*}
For the second term, we use the algebraic identity, which is a consequence of the Cayley-Hamilton theorem: 
\begin{align*}
I_k\sec+\sec I_k&=(\mathrm{tr}I_k)\sec+(\mathrm{tr}\sec) I_k+(\mathrm{tr}(I_k \sec)-(\mathrm{tr}I_k)(\mathrm{tr}\sec))I\\
&=(\mathrm{tr}I_k)\sec+\mathrm{tr}\sec (a_k\sec+b_kI+\mc{P}_k)+(\mathrm{tr}(I_k \sec)-(\mathrm{tr}I_k)(\mathrm{tr}\sec))I.
\end{align*}
Combining the two equations above, we obtain
\begin{multline}\label{260516e6_66}
I_{k+1}=(a_k'+2a_k\mathrm{tr}\sec+\mathrm{tr}I_k) \sec\\
 +(b_k'+b_k\mathrm{tr}\sec-a_k\det\sec+\mathrm{tr}(I_k \sec)-(\mathrm{tr}I_k)(\mathrm{tr}\sec))I\\
 +(\partial_s \mc{P}_k+(\mathrm{tr}\sec) \mc{P}_k+a_kR).
\end{multline}
This finishes the proofs of both \eqref{eq:decomposition} and \eqref{eq:coeffi recursion}.
\end{proof}

Let us record one more recursive relation that will be very useful later. Denote
\begin{equation}
    \mathcal{P}_k(s):= \mathcal{P}_k(R(s), \dots, R^{(k-3)}(s)). 
\end{equation}
From \eqref{260516e6_66} and \eqref{eq:decomposition}, we see that 
\begin{equation}\label{260506e4_94}
    \mathcal{P}_{k+1}(s)=
    \partial_s \mathcal{P}_k(s)+ (\mathrm{tr}(\sec(\epsilon, s))) \mathcal{P}_k(s)+ 
    a_k(s)R(s),
\end{equation}
where $a_k$ is the same as the one in \eqref{eq:coeffi recursion}. 
One can compute directly (see for \eqref{eq:I_k recursion}) that 
\begin{equation}
\mathcal{P}_1(s)=\mathcal{P}_2(s)=0.
\end{equation}
From this we can compute that 
\begin{equation}
\mathcal{P}_3(s)= a_2(s) R(s).
\end{equation}
From Proposition \ref{260505prop4_5}, we obtain that $a_2(s)=2$, and therefore 
\begin{equation}
\mathcal{P}_3(s)=2 R(s).
\end{equation}
We apply \eqref{260506e4_94} iteratively, and obtain that $\mathcal{P}_i(0)$ is a linear combination of $R(0), R'(0), \dots$, and that 
\begin{equation}\label{260516e6_73}
\mathcal{P}_m(0)= 2 R^{(m-3)}(0)+ \text{lower derivatives of } R,
\end{equation}
for all $m\ge 3$. \\

Next, we will compute the $\partial_s^k\det W(s,\epsilon)$ where
\begin{equation}
W(s,\epsilon)=\sec(s,\epsilon)-\sec(0,\epsilon).
\end{equation}
For a $2\times 2$ matrix
\begin{equation}
S=\left(\begin{array}{cc}
a & b \\
b & c
\end{array}\right)
\end{equation}
 we write the companion of $S$
as
\begin{equation}
S^{\sharp}=\left(\begin{array}{cc}
c & -b \\
-b & a
\end{array}\right)
\end{equation}
 Then we have the following Jacobi's formula
\begin{equation}\label{260505e4_36}
\frac{d}{ds}\det S(s)=\mathrm{tr} (S^{\sharp}S'), 
\end{equation}
where $S^{\sharp}$ is the companion of $S$. More importantly, we will see that coefficients of $\partial_s^k\det W(0,\epsilon)$ satisfy the same recursive relation with the coefficients $a_k$ in the decomposition \eqref{eq:decomposition} of $I_k$.
\begin{proposition}\label{260505prop4_6}
We can write
\begin{equation}\label{eq:derivative determinant}
\partial_s^k\det W(s,\epsilon)=d_k(s)(\det A_2(s))^{-2}+\mathrm{tr}(W^{\sharp}(s,\epsilon) \partial_s^{k}W(s,\epsilon)),
\end{equation}
where $d_k$ is a scalar function depending on $\epsilon$ as well, and we ignore the dependence just for simplicity. 
Moreover we have $d_1=0$ and a recursive relation
\begin{equation}\label{eq:determinant recursion}
d_{k+1}=d_k'+2d_k\mathrm{tr}\sec +\mathrm{tr}I_k.
\end{equation}
\end{proposition}
\begin{proof}[Proof of Proposition \ref{260505prop4_6}]
We will prove \eqref{eq:derivative determinant} by induction. The $k=1$ case is straightforward. Next, we will prove the $k+1$ case. We have
\begin{multline}
\partial_s^{k+1}\det W(s,\epsilon)=d_k'(\det A_2(s))^{-2}-2d_k(\det A_2)^{-3}(\det A_2)'\\
+\mathrm{tr}((W^{\sharp})' \partial_s^{k}W)+\mathrm{tr}(W^{\sharp}\partial_s^{k+1}W).
\end{multline}
We keep the first and fourth terms. For the second term, by Jacobi formula \eqref{260505e4_36} and \eqref{eq: Reverse sec}, we have
\begin{equation}
(\det A_2)^{-3}(\det A_2)'=(\det A_2)^{-2}\mathrm{tr}( A_2^{-1}A_2')=-\mathrm{tr}\sec(\det A_2)^{-2}.
\end{equation}
Note that
\begin{equation}
(S^{\sharp})'=(S')^{\sharp}
\end{equation}
 for any $2\times 2$ matrix $S$. For the third term, combining this with \eqref{eq:def of I_k} and $I_1=I$, we have 
$$
\mathrm{tr}((W^{\sharp})' \partial_s^{k}W)=\mathrm{tr}((W')^{\sharp} \partial_s^{k}W)=\mathrm{tr}(((A_2^T)^{-1}A_2^{-1})^{\sharp} (A_2^T)^{-1}I_kA_2^{-1}).
$$
By using  
\begin{equation}
S^{\sharp}=(\det S )S^{-1}
\end{equation}
 for any invertible matrix $S$, we have
\begin{multline}
((A_2^T)^{-1}A_2^{-1})^{\sharp} =((A_2A_2^T)^{-1})^{\sharp} \\
=\det ((A_2A_2^T)^{-1})A_2A_2^T=(\det A_2)^{-2}A_2A_2^T.
\end{multline}
Combining the two equations above, we obtain
\begin{equation}
\mathrm{tr}((W^{\sharp})' \partial_s^{k}W)=(\det A_2)^{-2}\mathrm{tr}(A_2A_2^T (A_2^T)^{-1}I_kA_2^{-1})=(\det A_2)^{-2}\mathrm{tr}I_k.
\end{equation}
Combining the four terms together, we have
$$
\partial_s^{k+1}\det W(s,\epsilon)=(d_k'+2d_k\mathrm{tr}\sec+\mathrm{tr}I_k)(\det A_2)^{-2}+\mathrm{tr}(W^{\sharp}\partial_s^{k+1}W).
$$
This finishes the proof of Proposition \ref{260505prop4_6}. 
\end{proof}

\begin{proposition}\label{260430prop6_7}
Let $k\geq 4$. Define $(c_1(\epsilon),\ldots, c_k(\epsilon))$ such that
\begin{align*}
\sum_{i=1}^k\partial_s^iW(0,\epsilon)c_i(\epsilon)&=0,\\
\sum_{i=1}^k\partial_s^i\det W(0,\epsilon)c_i(\epsilon)&=0,
\end{align*}
that is, $(c_1(\epsilon),\ldots, c_k(\epsilon))$ is an element in the kernel of the matrix
\begin{equation}\label{260430e6_18}
\left[\begin{array}{cccc}
\partial_s \mathcal{W}(0, \epsilon) & \partial_s^2 \mathcal{W}(0, \epsilon) & \ldots & \partial_s^k \mathcal{W}(0, \epsilon) \\
\partial_s W_{11}(0, \epsilon) & \partial_s^2 W_{11}(0, \epsilon) & \ldots & \partial_s^k W_{11}(0, \epsilon) \\
\partial_s W_{12}(0, \epsilon) & \partial_s^2 W_{12}(0, \epsilon) & \ldots & \partial_s^k W_{12}(0, \epsilon) \\
\partial_s W_{22}(0, \epsilon) & \partial_s^2 W_{22}(0, \epsilon) & \ldots & \partial_s^k W_{22}(0, \epsilon)
\end{array}\right]
\end{equation} 
If the dimension of 
\begin{equation}\label{260430e6_17}
    \mathrm{span}\{I,R_0,\ldots, R_0^{(k-3)}\}
\end{equation}
is $\leq 2$, then for every $\epsilon$, the dimension of the kernel is $\geq k-3$. In other words, the rank of the above matrix is $\leq 3$.

\end{proposition}
\begin{proof}[Proof of Proposition \ref{260430prop6_7}]
When the dimension of $\mathrm{span}\{I,R_0,\ldots, R_0^{(k-3)}\}=1$, the proof is similar and simpler. So, in the following, we assume that there is 
\begin{equation}
E_0\in\mathrm{span}\{I,R_0,\ldots, R_0^{(k-3)}\}
\end{equation}
such that for every $i=0,\ldots, k-3$,
\begin{equation}
R_0^{(i)}\in\mathrm{span}\{I,E_0\}.
\end{equation}
 So, by using Cayley-Hamilton theorem repetitively, or simply by using \eqref{260516e6_73}, we have\footnote{Indeed \eqref{260516e6_73} can be proven via applying Cayley-Hamilton theorem to \eqref{260505e4_28}; if we go back and check the proof of Proposition \ref{260505prop4_5} again, we will see that this is precisely how we obtained \eqref{260506e4_94}, and therefore \eqref{260516e6_73}. }
\begin{equation}
\mc{P}_i\in \mathrm{span}\{I,E_0\}, \text{ for any }i=1,\ldots, k.
\end{equation}
By the previous two propositions and $W^{\sharp}(0,\epsilon)=0$, we have that 
\begin{align*}
\sum_{i=1}^k\partial_s^iW(0,\epsilon)c_i(\epsilon)&=\sum_{i=1}^k
(A_2^T)^{-1} 
(c_ia_i \sec+c_ib_iI+c_i\mc{P}_i)(A_2)^{-1}\\
&=\sum_{i=1}^k
(A_2^T)^{-1}
(c_ia_i \sec+c_i\tilde{b}_iI+c_i\tilde{p}_iE_0)
(A_2)^{-1},
\end{align*}
for some new scalars $\tilde{b}_i$ and $\tilde{p}_i$, and 
\begin{equation}
\sum_{i=1}^k\partial_s^i\det W(0,\epsilon)c_i(\epsilon)=\sum_{i=1}^kc_id_i(\det A_2)^{-2}.
\end{equation}
Since the initial conditions and recursive relations of $a_i$ and $d_i$ coincide respectively, we know that 
\begin{equation}\label{aidi_equal}
a_i=d_i, \text{ for any } i , \epsilon.
\end{equation}
So, the linear system of finding $(c_1, \dots, c_k)$ can be reduced to
\begin{equation}
\sum_{i=1}^k\tilde{b}_ic_i=0,
\end{equation}
\begin{equation}
\sum_{i=1}^k\tilde{p}_ic_i=0
\end{equation}
and 
\begin{equation}
\sum_{i=1}^kd_ic_i=0.
\end{equation}
From these it is easy to see that the dimension of the kernel is $\ge k-3$. This finishes the proof of the proposition.  
\end{proof}

Proposition \ref{260430prop6_7} leads to a contradiction  that we are assuming item (3.2) in Theorem \ref{260509theorem2_1}. This finishes the proof that (3.2) implies (3.3).

\subsection{Item (3.3) implies item (3.2)}\label{260603subsection4_3}

Let $k \ge 4$.  We argue by contradiction, and assume that for every $\epsilon$ small enough, the first $k$ columns of the matrix \eqref{260309e1_4} has rank $\le 3$. Our goal is to derive a contradiction to item (3.3).\\

We continue with the setup in the proof that (3.2) implies (3.3). In particular, we will consider the rank of the matrix \eqref{260430e6_18}. Fix an $\epsilon>0$ small enough. There exist $c_i=c_i(\epsilon)\in \R, i=1, \dots, 4$, not all zero, such that 
\begin{multline}\label{260506e4_72}
c_1(\epsilon) \partial_s^i W_{11}(0, \epsilon) + c_2(\epsilon) \partial_s^i W_{12}(0, \epsilon) \\
+ c_3(\epsilon) \partial_s^i W_{22}(0, \epsilon) + c_4(\epsilon) \partial_s^i \det W(0, \epsilon) = 0, \ \forall i=1, 2, \dots, k.  
\end{multline}
Let us write the above relation in a more compact form: Denote 
\begin{equation}
K=
\begin{bmatrix}
c_1 & c_2/2\\
c_2/2 & c_3
\end{bmatrix}
\end{equation}
and we can write \eqref{260506e4_72} equivalently as 
\begin{equation}\label{260506e4_75}
\operatorname{tr}\big(K(\epsilon) \partial_s^i W(0, \epsilon)\big) + c_4 \partial_s^i \det W(0, \epsilon) = 0.
\end{equation}
Recall from \eqref{eq:def of I_k} that 
\begin{equation}
\partial^i_s W(0, \epsilon)= (A_2^T(0))^{-1} I_i(\epsilon, 0) (A_2(0))^{-1},
\end{equation}
and from Proposition \ref{260505prop4_6} that 
\begin{equation}
\partial^i_s \det W(0, \epsilon)=
a_i(0) (\det A_2(0))^{-2}= 
d_i(0) 
(\det A_2(0))^{-2}
\end{equation}
where we applied \eqref{aidi_equal} and 
\begin{equation}
W^{\sharp}(0, \epsilon)=0,
\end{equation}
we can write \eqref{260506e4_72} further as 
\begin{equation}\label{260506e4_79}
 \operatorname{tr}\big(K (A_2^T)^{-1} I_i(\epsilon, 0) (A_2)^{-1}\big) + c_4 a_i(0) \det(A_2)^{-2} = 0. 
\end{equation}
By using elementary properties of traces, the first term can be written as 
\begin{equation}
\operatorname{tr}\big(
(A_2)^{-1}
K (A_2^T)^{-1} I_i(\epsilon, 0) \big). 
\end{equation}
Let us denote 
\begin{equation}
Y= 
(A_2)^{-1} K(A_2^T)^{-1},
\end{equation}
and 
\begin{equation}
\lambda= c_4 \det(A_2)^{-2}.
\end{equation}
Under these notations, \eqref{260506e4_79} is equivalent to 
\begin{equation}\label{260506e4_83}
\mathrm{tr} (Y(\epsilon) I_i(\epsilon, 0))+ \lambda(\epsilon) a_i(0)=0, \ i=1, \dots, k.
\end{equation}
Note that 
\begin{equation}
(c_1, c_2, c_3, c_4) \neq 0
\end{equation}
 guarantees that 
 \begin{equation}\label{260506e4_85}
 (Y, \lambda) \neq (0, 0).
 \end{equation}
Now we consider some special values of $i$. First of all, we take $i=1$. We have that $I_1=I$, the $2\times 2$ identity matrix, and $a_1(0)=0$. As a consequence, we obtain that 
\begin{equation}
\mathrm{tr}(Y(\epsilon))=0,
\end{equation}
that is, $Y$ is a trace-free symmetric matrix. Next, we take $i=2$. In this case, from Proposition \ref{260504prop4_1}, we obtain that 
\begin{equation}\label{260506e4_87}
I_2(\epsilon, 0)=
2\sec(\epsilon, 0),
\end{equation}
and from Proposition \ref{260505prop4_5}, we obtain that $a_2(0)=2. $ By substituting these into \eqref{260506e4_83}, we obtain 
\begin{equation}
2\mathrm{tr}(Y(\epsilon) \sec(\epsilon, 0))+ 2 \lambda(\epsilon) =0,
\end{equation}
which implies that 
\begin{equation}\label{260506e4_89}
\lambda(\epsilon)= 
-\mathrm{tr}(Y(\epsilon) \sec(\epsilon, 0)).
\end{equation}
From \eqref{260506e4_89} we can also conclude that $Y\neq 0$, as otherwise both $Y$ and $\lambda$ need to vanish, which contradicts \eqref{260506e4_85}. \\

Now we go back to \eqref{260506e4_83} and consider what it says for a more general $i\ge 3$. We first follow \eqref{eq:decomposition}, and write 
\begin{equation}\label{260506e4_90}
I_i(\epsilon, 0)=
a_i(0) \sec(\epsilon, 0)
+ 
b_i(0) I+
\mathcal{P}_i(R(0), \dots, R^{(i-3)}(0)).
\end{equation}
We substitute this into \eqref{260506e4_83} with $i\ge 3$, and obtain 
\begin{equation}
a_i(0) \mathrm{tr}(Y(\epsilon) \sec(\epsilon, 0))
+
\mathrm{tr}(Y(\epsilon) \mathcal{P}_i)+ a_i(0) \lambda(\epsilon)=0,
\end{equation}
where we used the fact that $Y$ is trace-free. This, combined with \eqref{260506e4_89}, implies that 
\begin{equation}\label{260506e4_92}
\mathrm{tr}(Y(\epsilon) \mathcal{P}_i)=0, \ k\ge i\ge 3.
\end{equation}
The recursive relation \eqref{260516e6_73}, combined with \eqref{260506e4_92}, implies that 
\begin{equation}
\mathrm{tr}(Y(\epsilon) R^{(i)}(0))=0, \ i=0, 1, \dots, k-3.
\end{equation}
We claim that 
\begin{equation}\label{260506e4_100}
\mathrm{span}\{I, R(0), \dots, R^{(k-3)}(0)\}\le 2.
\end{equation}
To see \eqref{260506e4_100}, we just need to write 
\begin{equation}
R^{(i)}(0)= R^{(i)}(0)- \frac{\mathrm{tr} R^{(i)}(0)}{2} I+ \frac{\mathrm{tr} R^{(i)}(0)}{2} I.
\end{equation}
By the fact that $Y(\epsilon)$ is trace-free, we obtain that the trace of 
\begin{equation}
Y(\epsilon)\pnorm{
R^{(i)}(0)- \frac{\mathrm{tr} R^{(i)}(0)}{2} I
}
\end{equation}
also vanishes. From this and the fact that $Y\neq 0$, we can conclude \eqref{260506e4_100} immediately.\\

Once \eqref{260506e4_100} is proven, we see immediately that $R(0), \dots, R^{(k-3)}(0)$ are simultaneously diagonalizable, and therefore we arrive at a contradiction.

\subsection{The equivalence of (3.3), (3.4) and (3.5)}\label{260603subsection4_4}

We now prove that Items (3.3), (3.4) and (3.5) in Theorem~\ref{260509theorem2_1}
are equivalent. The proof is elementary, but we give the details because it is
important to keep track of the different notations for the Jacobi operators.\\

Let $\mathcal{R}$ be the Riemannian curvature tensor. Set
\begin{equation}
Z:=\gamma'(0),
\qquad
\bfP:=Z^\perp\subset T_{\gamma(0)}\mathcal{M}.
\end{equation}
For \(0\leq m\leq k-3\), define the linear operator
\(\mathcal R_m:T_{\gamma(0)}M\to T_{\gamma(0)}M\) by
\begin{equation}
\mathcal R_m(v):=(\nabla_Z^m \mathcal{R})(v,Z)Z.
\end{equation}
Thus, in the notation of Definition~\ref{def:k exceptional},
\begin{equation}
\mathcal R_m=\mathcal{R}^{m+2}_{\gamma(0)}.
\end{equation}
Moreover,
\begin{equation}
\mathcal R_m(Z)=0.
\end{equation}
The restriction of \(\mathcal R_m\) to \(\bfP\) is self-adjoint. If
\(E_1(t),E_2(t)\) is the parallel orthonormal frame used in
Definition~\ref{260504defi1_4}, then the matrix of
\(\mathcal R_m|_{\bfP}\) in the basis \(E_1(0),E_2(0)\) is exactly
\begin{equation}
R^{(m)}(0),
\end{equation}
where the matrix $R$ is defined in \eqref{260504e1_17}. 
Indeed,
\begin{multline}
\left\langle \mathcal R_m(E_i(0)),E_j(0)\right\rangle=
\left\langle
(\nabla_{\gamma'}^m \mathcal{R})(E_i,\gamma')\gamma',
E_j
\right\rangle\big|_{s=0}     \\
=
\frac{d^m}{ds^m}
\left\langle \mathcal{R}(E_i(s),\gamma'(s))\gamma'(s),E_j(s)\right\rangle
\bigg|_{s=0}           =
R_{ij}^{(m)}(0).
\end{multline}

We shall use the following elementary linear algebra fact. Let
\(A_0,\ldots,A_N\) be self-adjoint operators on a two-dimensional real inner
product space \(\bfP\). Then the following two statements are equivalent:

\begin{equation}
A_0,\ldots,A_N
\quad
\text{are simultaneously diagonalizable;}
\end{equation}
and there exists an orthonormal basis \((X,E)\) of \(\bfP\) such that
\begin{equation}
\langle A_mX,E\rangle=0,
\qquad
0\leq m\leq N.
\end{equation}
We first compare Item (3.3) with chaotic curvature. Let \(X=X(0)\in \bfP\) be a
unit vector, and let \(X(t)\) be its parallel transport along \(\gamma\). Let
\(E=E(0)\in \bfP\) be a unit vector perpendicular to \(X\), and let \(E(t)\) be its
parallel transport along \(\gamma\). Then
\begin{equation}
(X(t),E(t),\gamma'(t))
\end{equation}
is an orthonormal frame along \(\gamma\). By Definition~\ref{260505defi1_6},
\begin{equation}
Y(t)=\bar{\operatorname{Ric}}(X(t)).
\end{equation}
Since \(E(t)\) spans the orthogonal complement of
\(\operatorname{span}\{X(t),\gamma'(t)\}\), we have
\begin{equation}
Y^\perp(t)
=
\operatorname{Ric}_{\gamma(t)}(X(t),E(t))E(t).
\end{equation}
Using the curvature symmetries in dimension three, and the fact that
\(X(t),E(t)\) and \(\gamma'(t)\) are parallel along \(\gamma\), we obtain, for
\(0\leq m\leq k-3\),
\begin{multline}
\nabla_{\gamma'}^mY^\perp(0)
=
(\nabla_{\gamma'}^m\operatorname{Ric})(X,E)E       
\\
=
\left\langle
(\nabla_Z^m \mathcal{R})(X,Z)Z,E
\right\rangle E                                  
=
\left\langle
\mathcal R_m(X),E
\right\rangle E.
\end{multline}
Therefore
\begin{equation}
|Y^\perp(0)|+\cdots+
|\nabla_{\gamma'}^{k-3}Y^\perp(0)|\neq 0
\end{equation}
is equivalent to saying that
\begin{equation}
\left\langle
\mathcal R_m(X),E
\right\rangle\neq 0
\end{equation}
for at least one \(0\leq m\leq k-3\). This, combined with the linear algebra observation above, proves that (3.3) and (3.4) are equivalent. \\

It remains to compare Item (3.3) with Item (3.5). By Definition~\ref{def:k exceptional},
the geodesic \(\gamma\) is \((k-1)\)-exceptional if and only if there exists a
two-dimensional subspace
\begin{equation}
V\subset T_{\gamma(0)}\mathcal{M},
\qquad
Z\in V,
\end{equation}
such that
\begin{equation}
\mathcal R^q_{\gamma(0)}(V)\subset V,
\qquad
2\leq q\leq k-1.
\end{equation}
Writing \(q=m+2\), this is the same as
\begin{equation}
\mathcal R_m(V)\subset V,
\qquad
0\leq m\leq k-3.
\end{equation}
Since \(V\) is two-dimensional and contains \(Z\), there exists a unit vector
\(X\in \bfP\) such that
\begin{equation}
V=\operatorname{span}\{Z,X\}.
\end{equation}
Let \(E\in \bfP\) be a unit vector perpendicular to \(X\). Since
\(\mathcal R_m(Z)=0\) and \(\mathcal R_m(X)\in \bfP\), the condition
\(\mathcal R_m(V)\subset V\) is equivalent to
\begin{equation}
\mathcal R_m(X)\in \mathbb RX,
\end{equation}
or equivalently,
\begin{equation}
\left\langle
\mathcal R_m(X),E
\right\rangle=0,
\qquad
0\leq m\leq k-3.
\end{equation}
By the same linear algebra observation, this finishes the proof of the equivalence between (3.3) and (3.5).

\section{Proof of Theorem \ref{260509theorem2_1}: Item (2)}

Item (2) follows immediately from the following three  lemmas.

\begin{lemma}\label{260516lemma7_1}
    Assume that the matrix \eqref{260509e2_2} has rank three for all small enough $\epsilon$. Then $R(s)$ is simultaneously diagonalizable in $s$ and $R_{11}\not\equiv_s R_{22}$. 
\end{lemma}

\begin{lemma}\label{260516lemma7_2}
    Fix $k\ge 3$. Assume that $R(s)$ is simultaneously diagonalizable in $s$, and 
    \begin{equation}
        R_{11}^{(m)}(0)\neq R_{22}^{(m)}(0),
    \end{equation}
    for some $0\le m\le k-3$. Then the first $k$ columns of the $4\times \infty$ matrix \eqref{260509e2_2} has rank three. 
\end{lemma}

\begin{lemma}\label{260516lemma7_3}
    Fix $k\ge 3$. Assume that $R(s)$ is simultaneously diagonalizable in $s$, and assume that the first $k$ columns of the $4\times \infty$ matrix \eqref{260509e2_2} has rank three. Then \begin{equation}
        R_{11}^{(m)}(0)\neq R_{22}^{(m)}(0),
    \end{equation}
    for some $0\le m\le k-3$.
\end{lemma}

The proof of Lemma \ref{260516lemma7_3} is almost identical to that of Proposition \ref{260430prop6_7}, and we will not repeat it here. We will only present the proofs of Lemma \ref{260516lemma7_1} and Lemma \ref{260516lemma7_2}.\\

Let us first prove Lemma \ref{260516lemma7_1}. We argue by contradiction. If $R(s)$ is not simultaneously diagonalizable, then we know from item (3) of Theorem \ref{260509theorem2_1} that the matrix  \eqref{260509e2_2} has rank four for sufficiently small $\epsilon$, which is a contradiction. We therefore from now on assume that $R(s)$ is simultaneously diagonalizable. If we further assume that $R_{11}\equiv R_{22}$, then it is elementary to see that the matrix  \eqref{260509e2_2} has rank two, which is a contradiction. This finishes the proof of Lemma \ref{260516lemma7_1}.\\

Let us prove Lemma \ref{260516lemma7_2}.
We argue by contradiction, and  assume that for some $\epsilon$ small enough, the matrix \eqref{260430e6_18} has rank $\le 2$. Our goal is to prove that 
\begin{equation}
    R_{11}^{(m)}(0)=R_{22}^{(m)}(0), \ \forall 0\le m\le k-3. 
\end{equation}
By the assumption, we can find $c_1(\epsilon), c_2(\epsilon), c_3(\epsilon)$, not all zero, such that 
\begin{equation}\label{260507e5_2}
     c_1 \partial_s^i W_{11}(0, \epsilon) + c_2 \partial_s^i W_{22}(0, \epsilon) + c_3 \partial_s^i \det W(0, \epsilon) = 0. 
\end{equation}
Let $K=K(\epsilon)=\mathrm{diag}(c_1, c_2)$. Then \eqref{260507e5_2} can be written equivalently as 
\begin{equation}
     \operatorname{tr}\big(K \partial_s^i W(0, \epsilon)\big) + c_3 \partial_s^i \det W(0, \epsilon) = 0. 
\end{equation}
We repeat the previous proof (around \eqref{260506e4_83}) and obtain 
\begin{equation}\label{260506e4_83zz}
\mathrm{tr} (Y(\epsilon) I_i(\epsilon, 0))+ \lambda(\epsilon) a_i(0)=0, \ i=1, \dots, k.
\end{equation}
Note that $Y(\epsilon)$ is not only trace-free as before but also a diagonal matrix. This, combined with 
\begin{equation}
\mathrm{tr}(Y(\epsilon) R^{(i)}(0))=0, \ i=0, 1, \dots, k-3,
\end{equation}
implies the desired result. This finishes the proof of Lemma \ref{260516lemma7_2}. \\

Let us also remark here that after proving item (2) and item (3) in Theorem \ref{260509theorem2_1}, item (1) holds automatically.

\section{Proof of Theorem \ref{260601theorem2_1}}

Recall the setup in Subsection~\ref{260601subsection5_1}. For clarity in this
proof, write
\begin{equation}
\mathcal W(s,\epsilon):=\det W(s,\epsilon).
\end{equation}
Thus the first row of the matrix~\eqref{260509e2_2} consists of the derivatives
of \(\mathcal W\), while the last three rows consist of the derivatives of
\(W_{11},W_{12},W_{22}\).

Fix \(\epsilon>0\) sufficiently small. We argue by contradiction. Suppose that
the first row is a linear combination of the second, third and fourth rows. Then
there exist constants \(a,b,c\), depending possibly on \(\epsilon\), such that
\begin{equation}
\partial_s^j\mathcal W(0,\epsilon)
=
a\,\partial_s^jW_{11}(0,\epsilon)
+
2b\,\partial_s^jW_{12}(0,\epsilon)
+
c\,\partial_s^jW_{22}(0,\epsilon)
\end{equation}
for every \(j\geq 1\). Since
\begin{equation}
W(0,\epsilon)=0,
\qquad
\mathcal W(0,\epsilon)=0,
\end{equation}
the same identity also holds for \(j=0\). By analyticity in the \(s\)-variable,
we obtain
\begin{equation}
\mathcal W(s,\epsilon)
=
aW_{11}(s,\epsilon)+2bW_{12}(s,\epsilon)+cW_{22}(s,\epsilon)
\end{equation}
for every \(0\leq s<\epsilon\).\\

Now set
\begin{equation}
D:=
\begin{pmatrix}
c & -b\\
-b & a
\end{pmatrix}.
\end{equation}
Since \(W(s,\epsilon)\) is symmetric, we may write
\begin{equation}
W(s,\epsilon)=
\begin{pmatrix}
W_{11}(s,\epsilon) & W_{12}(s,\epsilon)\\
W_{12}(s,\epsilon) & W_{22}(s,\epsilon)
\end{pmatrix}.
\end{equation}
A direct calculation gives
\begin{multline}
\det\bigl(W(s,\epsilon)-D\bigr)
=
\det W(s,\epsilon)
-
aW_{11}(s,\epsilon)
\\
-
2bW_{12}(s,\epsilon)
-
cW_{22}(s,\epsilon)
+
\det D.
\end{multline}
Using the preceding linear identity, this becomes
\begin{equation}
\det\bigl(W(s,\epsilon)-D\bigr)=\det D
\end{equation}
for every \(0\leq s<\epsilon\). Multiplying by \((\epsilon-s)^2\), we get
\begin{equation}\label{260603e6_9}
\det\bigl((\epsilon-s)W(s,\epsilon)-(\epsilon-s)D\bigr)
=
(\epsilon-s)^2\det D.
\end{equation}
It remains to understand the behavior of \(W(s,\epsilon)\) as \(s\to\epsilon^-\).
We record the required form of Ledger's recursion formula. Put
\begin{equation}
r:=\epsilon-s
\end{equation}
and, for fixed \(s\), define
\begin{equation}
S_s(r):=\sec(s,s+r),
\end{equation}
the second fundamental form that was used in Lemma \ref{260601lemma5_4}. 
\begin{claim}\label{260603claim6_1}
    \(S_s(r)\) satisfies the Riccati equation
\begin{equation}
\partial_r S_s(r)+S_s(r)^2+R(s+r)=0.
\end{equation}
\end{claim}
\begin{proof}[Proof of Claim \ref{260603claim6_1}]
    This follows by applying Lemma \ref{260601lemma5_4}, differentiating
\begin{equation}
S_s(r)=B(s,s+r)^{-1}\partial_tB(s,s+r)
\end{equation}
and using the Jacobi equation
\begin{equation}
\partial_t^2B(s,t)+B(s,t)R(t)=0.
\end{equation} 
We leave out the rest of the details. 
\end{proof}

Write the Taylor expansion of \(R\) at \(s\) as
\begin{equation}
R(s+r)=\sum_{m=0}^{\infty}\frac{r^m}{m!}R^{(m)}(s).
\end{equation}
Ledger's expansion has the form
\begin{equation}
S_s(r)=\frac{1}{r}I+\sum_{k=0}^{\infty}r^k S_k(s).
\end{equation}
Substituting this expansion into the Riccati equation and comparing powers of
\(r\), we obtain first
\begin{equation}
S_0(s)=0,
\end{equation}
and then, for every \(k\geq 0\),
\begin{equation}
S_{k+1}(s)
=
-\frac{1}{k+3}
\left(
\frac{R^{(k)}(s)}{k!}
+
\sum_{\ell=0}^{k}S_\ell(s)S_{k-\ell}(s)
\right).
\end{equation}
In particular,
\begin{equation}
S_1(s)=-\frac{1}{3}R(s).
\end{equation}
Therefore, as \(r=\epsilon-s\to0^+\),
\begin{equation}
\sec(s,\epsilon)
=
\frac{1}{\epsilon-s}I
-
\frac{\epsilon-s}{3}R(s)
+
O\bigl((\epsilon-s)^2\bigr).
\end{equation}
Since \(\sec(0,\epsilon)\) is fixed as \(s\to\epsilon^-\), and since
\begin{equation}
W(s,\epsilon)=\sec(s,\epsilon)-\sec(0,\epsilon),
\end{equation}
we obtain
\begin{equation}
(\epsilon-s)W(s,\epsilon)
=
I-(\epsilon-s)\sec(0,\epsilon)+O\bigl((\epsilon-s)^2\bigr).
\end{equation}
Consequently,
\begin{equation}
\lim_{s\to\epsilon^-}(\epsilon-s)W(s,\epsilon)=I.
\end{equation}
It follows that
\begin{equation}
\lim_{s\to\epsilon^-}
\det\bigl((\epsilon-s)W(s,\epsilon)-(\epsilon-s)D\bigr)
=
\det I
=
1.
\end{equation}
On the other hand, from the identity \eqref{260603e6_9}, we obtain 
\begin{equation}
\lim_{s\to\epsilon^-}
\det\bigl((\epsilon-s)W(s,\epsilon)-(\epsilon-s)D\bigr)
=
0.
\end{equation}
This contradiction proves that the first row of~\eqref{260509e2_2} is not a
linear combination of the second, third and fourth rows.

\section{Proof of Theorem \ref{existence}}

\subsection{Proof of the first item}\label{algebraic lemma}

We first prove a purely algebraic lemma.

\begin{lemma}\label{lem:algebraic-ricci}
Let \((V,g)\) be a three-dimensional real inner product space. Let \(A\) be a
symmetric \((0,2)\)-tensor on \(V\), and let \(B\) be a \((0,3)\)-tensor on \(V\)
which is symmetric in its last two variables, namely
\begin{equation}
A(X,Y)=A(Y,X), \qquad B(Z,X,Y)=B(Z,Y,X).
\end{equation}
Then there exists an orthonormal frame \((X_0,E_0,Z_0)\) of \(V\) such that
\begin{equation}
A(X_0,E_0)=0, \qquad B(Z_0,X_0,E_0)=0.
\end{equation}
\end{lemma}

\begin{proof}[Proof of Lemma \ref{lem:algebraic-ricci}]
Let \(S:V\to V\) be the self-adjoint endomorphism associated to \(A\), so that
\begin{equation}
A(X,Y)=g(SX,Y).
\end{equation}
We first assume that \(S\) has three distinct eigenvalues. Choose an oriented
orthonormal eigenbasis \((e_1,e_2,e_3)\) of \(S\), and write
\begin{equation}
Se_i=\lambda_i e_i, \qquad \lambda_i\neq \lambda_j \quad \text{if } i\neq j.
\end{equation}
For \(X=x_1e_1+x_2e_2+x_3e_3\), set
\begin{equation}
\mathcal E:=\bigcup_{i=1}^3 \mathbb R e_i.
\end{equation}
If \(X\notin \mathcal E\), then \(SX\) is not colinear with \(X\). Define
\begin{equation}
\widetilde E:=X\times SX, \qquad \widetilde Z:=X\times \widetilde E,
\end{equation}
where \(\times\) denotes the cross product with respect to the chosen oriented
orthonormal basis. Then \(X,\widetilde E,\widetilde Z\) are mutually orthogonal,
and
\begin{equation}
A(X,\widetilde E)=g(SX,\widetilde E)=0.
\end{equation}
Thus, after normalization, \((X,\widetilde E,\widetilde Z)\) gives an
orthonormal frame satisfying the first desired identity. It remains to choose
\(X\notin \mathcal E\) so that
\begin{equation}
B(\widetilde Z,X,\widetilde E)=0.
\end{equation}

We now write this condition explicitly. Up to harmless nonzero normalizing
factors, we have
\begin{equation}
\widetilde E
=
\bigl(
(\lambda_3-\lambda_2)x_2x_3,\,
(\lambda_1-\lambda_3)x_3x_1,\,
(\lambda_2-\lambda_1)x_1x_2
\bigr),
\end{equation}
and
\begin{equation}
\begin{aligned}
\widetilde Z
={}&\bigl(
(\lambda_3-\lambda_1)x_1x_3^2
  +(\lambda_2-\lambda_1)x_1x_2^2,  \\
&\quad
(\lambda_1-\lambda_2)x_2x_1^2
  +(\lambda_3-\lambda_2)x_2x_3^2,  \\
&\quad
(\lambda_2-\lambda_3)x_3x_2^2
  +(\lambda_1-\lambda_3)x_3x_1^2
\bigr).
\end{aligned}
\end{equation}
Define the homogeneous polynomial
\begin{equation}
F(X):=B(\widetilde Z,X,\widetilde E).
\end{equation}
If \(F\) vanishes at some point \(X\notin \mathcal E\), then normalizing
\(X,\widetilde E,\widetilde Z\) gives the desired frame. We therefore suppose,
for contradiction, that \(F\) does not vanish on \(\mathbb R^3\setminus\mathcal E\).
Since \(\mathbb R^3\setminus\mathcal E\) is path connected, \(F\) has a fixed
sign there. Replacing \(B\) by \(-B\) if necessary, we may assume
\begin{equation}
F(X)\geq 0 \qquad \text{for every } X\in \mathbb R^3.
\end{equation}

Write
\begin{equation}
B_{ijk}:=B(e_i,e_j,e_k).
\end{equation}
Thus \(B_{ijk}=B_{ikj}\). Setting \(x_1=0\), we obtain
\begin{equation}
\begin{aligned}
F(0,x_2,x_3)
={}&
(\lambda_3-\lambda_2)^2x_2^2x_3^2  \\
&\quad\cdot
\Bigl(
-B_{321}x_2^2+B_{231}x_3^2
 +(B_{221}-B_{331})x_2x_3
\Bigr).
\end{aligned}
\end{equation}
Since \(F\geq 0\), the quadratic polynomial in parentheses is nonnegative for
all \((x_2,x_3)\). Hence
\begin{equation}
B_{321}\leq 0,\qquad B_{231}\geq 0,\qquad
(B_{221}-B_{331})^2+4B_{231}B_{321}\leq 0.
\end{equation}
On the other hand, if we regard \(F\) as a polynomial in \(x_1\), then its
coefficient of \(x_1^4\) is
\begin{equation}
\begin{aligned}
&-B_{213}(\lambda_1-\lambda_2)^2x_2^2
+B_{312}(\lambda_1-\lambda_3)^2x_3^2  \\
&\qquad
+(B_{212}-B_{313})(\lambda_1-\lambda_3)(\lambda_1-\lambda_2)x_2x_3 .
\end{aligned}
\end{equation}
Again using \(F\geq 0\), this coefficient must be nonnegative for all
\((x_2,x_3)\). Therefore
\begin{equation}
B_{213}\leq 0,\qquad B_{312}\geq 0,\qquad
(B_{212}-B_{313})^2+4B_{213}B_{312}\leq 0.
\end{equation}
Using the symmetry \(B_{ijk}=B_{ikj}\), we have
\begin{equation}
B_{321}=B_{312},\qquad B_{231}=B_{213},\qquad
B_{221}=B_{212},\qquad B_{331}=B_{313}.
\end{equation}
The preceding inequalities therefore imply
\begin{equation}
B_{312}=B_{213}=0,\qquad B_{221}=B_{331}.
\end{equation}
Consequently, both the constant term and the \(x_1^4\)-coefficient of \(F\),
viewed as a polynomial in \(x_1\), vanish. Repeating the same argument after
cyclically permuting the indices \(1,2,3\), we see that the constant term and
the \(x_i^4\)-coefficient vanish when \(F\) is regarded as a polynomial in
\(x_i\), for each \(i=1,2,3\). In particular,
\begin{equation}
B_{ijk}=0
\qquad
\text{whenever } i,j,k \text{ are pairwise distinct}.
\end{equation}

We now finish the argument. Fixing \(x_2,x_3\), the polynomial \(F\), as a
polynomial in \(x_1\), is nonnegative for all \(x_1\), has zero constant term,
and has zero \(x_1^4\)-coefficient. Hence its \(x_1\)- and \(x_1^3\)-coefficients
must vanish. The same argument applies to \(x_2\) and \(x_3\). Therefore the
only possible remaining monomial in \(F\) is a multiple of
\begin{equation}
x_1^2x_2^2x_3^2.
\end{equation}
But the coefficient of this monomial only involves the components \(B_{ijk}\)
with \(i,j,k\) pairwise distinct, and all such components have already been
shown to vanish. Thus
\begin{equation}
F\equiv 0.
\end{equation}
In particular, choosing any \(X\notin\mathcal E\) gives \(F(X)=0\), and hence
gives the desired frame.

It remains to remove the assumption that \(S\) has distinct eigenvalues. Let
\(A_n\) be a sequence of symmetric \((0,2)\)-tensors whose associated
self-adjoint endomorphisms have distinct eigenvalues, and such that
\begin{equation}
A_n\to A.
\end{equation}
By the distinct-eigenvalue case, for each \(n\) there exists an orthonormal
frame \((X_n,E_n,Z_n)\) such that
\begin{equation}
A_n(X_n,E_n)=0,\qquad B(Z_n,X_n,E_n)=0.
\end{equation}
The space of orthonormal frames is compact, so after passing to a subsequence
we may assume
\begin{equation}
(X_n,E_n,Z_n)\to (X_0,E_0,Z_0).
\end{equation}
Passing to the limit gives
\begin{equation}
A(X_0,E_0)=0,\qquad B(Z_0,X_0,E_0)=0.
\end{equation}
This proves the lemma.
\end{proof}

We now prove Theorem~\ref{existence}, Item~(1). We prove the slightly stronger
pointwise statement that, for every \(p\in \mathcal{M}\), there exists a unit speed geodesic
\(\gamma\) with \(\gamma(0)=p\) and a unit vector
\(X(0)\in T_p \mathcal{M}\), \(X(0)\perp \dot\gamma(0)\), such that
\begin{equation}
Y^\perp(0)=0,\qquad \nabla_{\dot\gamma}Y^\perp(0)=0,
\end{equation}
where \(Y(t)=\bar{\operatorname{Ric}}(X(t))\) and \(Y^\perp(t)\) is as in
Definition~\ref{260505defi1_6}. This shows that the chaotic curvature condition
of order \(\leq 1\) fails at \(p\) along the geodesic \(\gamma\).

Fix \(p\in \mathcal{M}\), and apply Lemma~\ref{lem:algebraic-ricci} to \(V=T_p \mathcal{M}\) with
\begin{equation}
A(U,V):=\operatorname{Ric}_p(U,V),
\qquad
B(Z,U,V):=(\nabla_Z\operatorname{Ric})_p(U,V).
\end{equation}
The tensor \(A\) is symmetric, and \(B\) is symmetric in its last two variables
because the Ricci tensor is symmetric. Hence there exists an orthonormal frame
\((X_0,E_0,Z_0)\) of \(T_p \mathcal{M}\) such that
\begin{equation}
\operatorname{Ric}_p(X_0,E_0)=0,
\qquad
(\nabla_{Z_0}\operatorname{Ric})_p(X_0,E_0)=0.
\end{equation}

Let \(\gamma\) be the unit speed geodesic with
\begin{equation}
\gamma(0)=p,\qquad \dot\gamma(0)=Z_0.
\end{equation}
Let \(X(t)\) and \(E(t)\) be the parallel transports of \(X_0\) and \(E_0\)
along \(\gamma\), respectively. Thus
\begin{equation}
X(0)=X_0,\qquad E(0)=E_0,\qquad
\nabla_{\dot\gamma}X=\nabla_{\dot\gamma}E=0.
\end{equation}
By Definition~\ref{260505defi1_6},
\begin{equation}
Y(t)=\bar{\operatorname{Ric}}(X(t)).
\end{equation}
Since the orthogonal complement of
\(\operatorname{span}\{X(t),\dot\gamma(t)\}\) is spanned by \(E(t)\), we have
\begin{equation}
\begin{aligned}
Y^\perp(t)
&=
\big\langle \bar{\operatorname{Ric}}(X(t)),E(t)\big\rangle E(t)  \\
&=
\operatorname{Ric}_{\gamma(t)}(X(t),E(t))E(t).
\end{aligned}
\end{equation}
Therefore
\begin{equation}
Y^\perp(0)=\operatorname{Ric}_p(X_0,E_0)E_0=0.
\end{equation}
Moreover, since \(X(t)\) and \(E(t)\) are parallel along \(\gamma\), we have
\begin{equation}
\nabla_{\dot\gamma}Y^\perp(t)
=
(\nabla_{\dot\gamma(t)}\operatorname{Ric})(X(t),E(t))E(t).
\end{equation}
Evaluating at \(t=0\), we get
\begin{equation}
\nabla_{\dot\gamma}Y^\perp(0)
=
(\nabla_{Z_0}\operatorname{Ric})_p(X_0,E_0)E_0
=
0.
\end{equation}
Thus
\begin{equation}
|Y^\perp(0)|+\left|\nabla_{\dot\gamma}Y^\perp(0)\right|=0,
\end{equation}
so the chaotic curvature condition of order \(\leq 1\) fails at the point \(p\)
along the geodesic \(\gamma\). Since \(p\in \mathcal{M}\) was arbitrary, no
three-dimensional Riemannian manifold satisfies the chaotic curvature condition
of order \(\leq 1\). This proves Item~(1) of Theorem~\ref{existence}.

\subsection{Proof of the second item}\label{260606subsec7_2}

We mainly follow the notation from \cite{DGGZ24}, in particular, the notation in \cite[Subsection 3.1]{DGGZ24}. We will first find a metric satisfying the chaotic curvature condition of order $\le 2$; by compactness and continuity, this immediately implies that every small smooth perturbation of it also satisfies chaotic curvature condition of order $\le 2$. Next we will find another metric which fails the chaotic curvature condition of $\le 2$, and so does every small smooth perturbation of it. \\

Let us work in a small neighborhood of $0\in \R^3$, with coordinates 
\begin{equation}
    (x, y, z)=(x^1, x^2, x^3). 
\end{equation}
Let $g$ be a given metric. 
We use 
\begin{equation}
    e_i= \partial x^i
\end{equation}
to denote one vector in the standard basis for the tangent space. Denote 
\begin{equation}
    g_{ij}:= g(e_i, e_j),
\end{equation}
and 
\begin{equation}
    (g^{kl})_{1\le k, l\le 3}= g^{-1}.
\end{equation}
Moreover, denote 
\begin{equation}
    g_{ij, k}= \partial_k g_{ij}.
\end{equation}
In local coordinates, Christoffel symbols are given by 
\begin{equation}
    \Gamma_{ij}^k = \frac{1}{2} \sum_{m=1}^3 g^{km}(g_{mi,j} + g_{mj,i} - g_{ij,m}),
\end{equation}
and components of the Riemannian curvature tensor are given by 
\begin{equation}
    {R_{lki}}^q = -\left(\partial_k \Gamma_{li}^q - \partial_l \Gamma_{ki}^q - \sum_{p=1}^3 \Gamma_{lp}^q \Gamma_{ki}^p + \sum_{p=1}^3 \Gamma_{kp}^q \Gamma_{li}^p\right).
\end{equation}
Denote 
\begin{equation}
    \mathrm{Ric}_{ij}:= \mathrm{Ric}(e_i, e_j),
\end{equation}
which can be computed by 
\begin{equation}
    \sum_{k=1}^3 
    g(
    \mathcal{R}(e_k, e_i)e_j, e_k
    )=
    \sum_{k=1}^3
    \sum_{l}
    g(R_{kij}^{\ \ \ l} e_l, e_k)= \sum_k R_{kij}^{\ \ \ k}. 
\end{equation}
Next, let us compute the first covariant derivative 
\begin{equation}
    (\nabla_Z \mathrm{Ric})(X, Y), 
\end{equation}
which, by definition, is equal to 
\begin{equation}
    Z(\mathrm{Ric}(X, Y))- 
    \mathrm{Ric}(\nabla_Z X, Y)-
    \mathrm{Ric}(X, \nabla_Z Y). 
\end{equation}
Denote 
\begin{equation}
    \mathrm{Ric}_{ij; k}:= (\nabla_{e_k} \mathrm{Ric})(e_i, e_j).
\end{equation}
In local coordinates, we can compute that 
\begin{equation}
    \mathrm{Ric}_{ij; k}=
    \nabla_k \text{Ric}_{ij} = \partial_k \text{Ric}_{ij} - \sum_{m=1}^3 \Gamma_{ki}^m \text{Ric}_{mj} - \sum_{m=1}^3 \Gamma_{kj}^m \text{Ric}_{im}.
\end{equation}
In the end, we compute the second covariant derivative 
\begin{equation}
    (\nabla_Z \nabla_Z \text{Ric})(X, Y).
\end{equation}
Denote 
\begin{equation}
    \mathrm{Ric}_{ij; kl}:= (\nabla_{e_l}\nabla_{e_k} \mathrm{Ric})(e_i, e_j).
\end{equation}
We can compute that 
\begin{multline}
\mathrm{Ric}_{ij; kl}=\partial_l \text{Ric}_{ij; k} \\
- \sum_{m=1}^3 \Gamma_{lk}^m \text{Ric}_{ij; m} - \sum_{m=1}^3 \Gamma_{li}^m \text{Ric}_{mj; k} - \sum_{m=1}^3 \Gamma_{lj}^m \text{Ric}_{im; k}. 
\end{multline}
So far we have finished writing down the quantities we need to compute. Our goal is to find a metric $g$ such that no matter which orthonormal basis $(X, Y, Z)$ at the origin we pick, the three quantities
\begin{equation}
    \mathrm{Ric}(X, Y), \ (\nabla_Z \mathrm{Ric})(X, Y), \ (\nabla_Z\nabla_Z \mathrm{Ric})(X, Y)
\end{equation}
never vanish at the same time. \\

We will work with the metric 
\begin{equation}
g=g_{11} d x^2+2 g_{12} d x d y+g_{22} d y^2+d z^2
\end{equation}
where 
\begin{equation}
\begin{gathered}
g_{11}=1-y^2+2 x y z-\frac{2}{3} x y^2 z-\frac{2}{9} x z^3+\frac{1}{3} y^4-\frac{3}{5} y^2 z^2-\frac{1}{3} z^4 \\
g_{12}=-\frac{1}{3} z^3-\frac{1}{3} x^2 y z+\frac{4}{5} x y z^2-\frac{1}{3} y^3 z-\frac{1}{9} y z^3 \\
g_{22}=1+z^2.
\end{gathered}
\end{equation}
To simplify our notation, let us denote 
\begin{equation}
    A(X, Y)= \mathrm{Ric}(X, Y),
\end{equation}
\begin{equation}
    B(Z, X, Y)= (\nabla_Z \mathrm{Ric})(X, Y),
\end{equation}
and 
\begin{equation}
    C(W, Z, X, Y)= (\nabla_W \nabla_Z \mathrm{Ric})(X, Y)
\end{equation}
where everything is evaluated at the origin. Moreover, denote 
\begin{equation}
    B_j(X, Y)= B(e_j, X, Y), \ j=1, 2, 3, 
\end{equation}
and 
\begin{equation}
    C_{jk}(X, Y)= C(e_j, e_k, X, Y), \ 1\le j, k\le 3.
\end{equation}
By a direct calculation, we obtain 
\begin{equation}
A=\left(\begin{array}{ccc}
1 & 0 & 0 \\
0 & 0 & 0 \\
0 & 0 & -1
\end{array}\right)
\end{equation}
\begin{equation}
B_1=\left(\begin{array}{ccc}
0 & 0 & 0 \\
0 & 0 & -1 \\
0 & -1 & 0
\end{array}\right), \quad B_2=0, \quad B_3=\left(\begin{array}{ccc}
0 & 1 & 0 \\
1 & 0 & 0 \\
0 & 0 & 0
\end{array}\right),
\end{equation}
and 
\begin{equation}
\begin{gathered}
C_{11}=\left(\begin{array}{ccc}
0 & 0 & -\frac{1}{3} \\
0 & 0 & 0 \\
-\frac{1}{3} & 0 & 0
\end{array}\right), \quad C_{12}=\left(\begin{array}{ccc}
0 & -\frac{4}{5} & 0 \\
-\frac{4}{5} & 0 & \frac{1}{3} \\
0 & \frac{1}{3} & 0
\end{array}\right), \\
C_{13}=\left(\begin{array}{ccc}
\frac{2}{3} & 0 & \frac{4}{5} \\
0 & 0 & 0 \\
\frac{4}{5} & 0 & \frac{2}{3}
\end{array}\right), \\
C_{21}=\left(\begin{array}{ccc}
0 & \frac{1}{5} & 0 \\
\frac{1}{5} & 0 & \frac{1}{3} \\
0 & \frac{1}{3} & 0
\end{array}\right), \quad C_{22}=\left(\begin{array}{ccc}
\frac{6}{5} & 0 & -1 \\
0 & 0 & 0 \\
-1 & 0 & \frac{6}{5}
\end{array}\right), \\
C_{23}=\left(\begin{array}{ccc}
0 & \frac{1}{3} & 0 \\
\frac{1}{3} & 0 & 1 \\
0 & 1 & 0
\end{array}\right), \\
C_{31}=\left(\begin{array}{ccc}
\frac{2}{3} & 0 & \frac{4}{5} \\
0 & 0 & 0 \\
\frac{4}{5} & 0 & \frac{2}{3}
\end{array}\right), \quad C_{32}=\left(\begin{array}{ccc}
0 & \frac{1}{3} & 0 \\
\frac{1}{3} & 0 & 0 \\
0 & 0 & 0
\end{array}\right), \\
C_{33}=\left(\begin{array}{ccc}
\frac{24}{5} & 0 & -\frac{1}{3} \\
0 & \frac{24}{5} & 0 \\
-\frac{1}{3} & 0 & 8
\end{array}\right) .
\end{gathered}
\end{equation}
We will show that no matter which orthonormal basis $(X, Y, Z)$ we pick, the three quantities 
\begin{equation}
    A(X, Y), \ B(Z, X, Y), \ C(Z, Z, X, Y)
\end{equation}
never vanish at the same time. \\

We introduce three auxiliary quantities
\begin{equation}
\begin{gathered}
q_A=A(X, X)-A(Y, Y)+2 i A(X, Y), \\
q_B=B(Z, X, X)-B(Z, Y, Y)+2 i B(Z, X, Y), \\
q_C=C(Z, Z, X, X)-C(Z, Z, Y, Y)+2 i C(Z, Z, X, Y) .
\end{gathered}
\end{equation}
If it happens that 
\begin{equation}
A(X, Y)=B(Z, X, Y)=C(Z, Z, X, Y)=0,
\end{equation}
then the three quantities $q_A, q_B, q_C$ are all real, and the three new quantities 
\begin{equation}
D_0=\operatorname{Im}\left(q_B \overline{q_C}\right), \quad D_1=-\operatorname{Im}\left(q_A \overline{q_C}\right), \quad D_2=\operatorname{Im}\left(q_A \overline{q_B}\right)
\end{equation}
must vanish at the same time. 

\begin{claim}\label{260602claim7_2}
    When we rotate $(X, Y)$ in the space $Z^{\perp}$, the three quantities $D_0, D_1, D_2$ stay unchanged. 
\end{claim}
\begin{proof}[Proof of Claim \ref{260602claim7_2}]
We first compute $A$ and $q_A$ under rotation. 
    Let us write the rotation as 
    \begin{equation}
\begin{gathered}
X^{\prime}=(\cos \theta) X+(\sin \theta) Y \\
Y^{\prime}=(-\sin \theta) X+(\cos \theta) Y
\end{gathered}
\end{equation}
To simplify our notation, let us denote 
\begin{equation}
c=\cos \theta, \quad s=\sin \theta
\end{equation}
and 
\begin{equation}
\alpha=A(X, X), \quad \beta=A(X, Y), \quad \gamma=A(Y, Y) .
\end{equation}
By the fact that $A$ is symmetric, we obtain 
\begin{equation}
A\left(X^{\prime}, X^{\prime}\right)=A(c X+s Y, c X+s Y)=c^2 \alpha+2 c s \beta+s^2 \gamma .
\end{equation}
Similarly, we obtain 
\begin{equation}
A\left(Y^{\prime}, Y^{\prime}\right)=A(-s X+c Y,-s X+c Y)=s^2 \alpha-2 c s \beta+c^2 \gamma,
\end{equation}
and 
\begin{equation}
A\left(X^{\prime}, Y^{\prime}\right)=\left(c^2-s^2\right) \beta+c s(\gamma-\alpha) .
\end{equation}
Let us compute $q_A(X', Y')$, and we obtain 
\begin{equation}
\begin{aligned}
q_A(X', Y')=& \left(c^2 \alpha+2 c s \beta+s^2 \gamma\right)-\left(s^2 \alpha-2 c s \beta+c^2 \gamma\right) \\
& \quad+2 i\left(\left(c^2-s^2\right) \beta+c s(\gamma-\alpha)\right)
\end{aligned}
\end{equation}
On the other hand, 
\begin{equation}
q_S(X, Y)=(\alpha-\gamma)+2 i \beta .
\end{equation}
Direct calculation shows that 
\begin{equation}
q_A\left(X^{\prime}, Y^{\prime}\right)=e^{-2 i \theta} q_A(X, Y).
\end{equation}
Similarly, we also obtain 
\begin{equation}
\begin{aligned}
& q_B\left(X^{\prime}, Y^{\prime}\right)=e^{-2 i \theta} q_B(X, Y), \\
& q_C\left(X^{\prime}, Y^{\prime}\right)=e^{-2 i \theta} q_C(X, Y) .
\end{aligned}
\end{equation}
The claim follows directly from the definition of $D_0, D_1, D_2$. 
\end{proof}

Recall that our goal is to show that no matter which orthonormal basis $(X, Y, Z)$ we pick, the three quantities 
\begin{equation}
    D_0, D_1, D_2
\end{equation}
cannot vanish at the same time. Write 
\begin{equation}
    Z=(a, b, c), \ \ a^2+ b^2+ c^2=1.
\end{equation}
Let us first work with the case that 
\begin{equation}
    r=\sqrt{a^2+ c^2}>0.
\end{equation}
By Claim \ref{260602claim7_2}, we can pick one fixed pair $(X, Y)$. Let us pick  
\begin{equation}
X=\left(\frac{c}{r}, 0,-\frac{a}{r}\right), \quad Y=\left(-\frac{a b}{r}, r,-\frac{b c}{r}\right) .
\end{equation}
Direct calculation shows that 
\begin{equation}
    D_0= 4 a c(1+ b^2), \ D_1= 2b(1+ b^2), \ D_2= 2(1+ b^2)(a^2-c^2). 
\end{equation}
It is elementary to see that these three quantities cannot vanish at the same time. \\

In the end, we still need to consider the case $r=0$. Similar calculation shows that 
\begin{equation}
    D_0=0,\  |D_1|=4, \ D_2=0.
\end{equation}
This finishes the proof of all the cases, and thus the proof that the metric we constructed satisfies the chaotic curvature condition of order $\le 2$ at the origin along geodesics pointing in every direction. By continuity and compactness, we see that  the chaotic curvature condition of order $\le 2$  holds for neighboring points of the origin as well, and for small smooth perturbations of the metric as well. \\

Next, we will find a metric which fails the chaotic curvature condition of order $\le 2$, and so does every small smooth perturbation of it. It is not difficult to imagine that there are plenty of these examples. More precisely, if an orthonormal frame $X, Y, Z$ is a ``transverse" solution to 
\begin{equation}
    \mathrm{Ric}(X, Y)=(\nabla_Z \mathrm{Ric})(X, Y)=(\nabla_Z\nabla_Z \mathrm{Ric})(X, Y)=0
\end{equation}
at the origin, then the underlying metric works for our purpose. Here ``transverse" means roughly that implicit function theorems can be applied to find solutions that are small perturbations of the frame $X, Y, Z$. Let us carry out the details. \\

A very clean example is a conformal perturbation of the Euclidean metric. Denote 
\begin{equation}
u(x, y, z)=\frac{1}{2}\left(-x^2+z^2\right)-\frac{1}{2} y z^2+\frac{1}{6} x z^3,
\end{equation}
and consider the metric 
\begin{equation}\label{260602e7_87}
    g= e^{2 u(x, y, z)} (d x^2+ d y^2+ d z^2).
\end{equation}
First of all, one can repeat the calculation in the first half of this subsection, and show that if we take 
\begin{equation}\label{260602e7_88}
    (X, Y, Z)=(e_1, e_2, e_3),
\end{equation}
then \begin{equation}
    \mathrm{Ric}(X, Y)=(\nabla_Z \mathrm{Ric})(X, Y)=(\nabla_Z\nabla_Z \mathrm{Ric})(X, Y)=0
\end{equation}
at the origin. It remains to show that this zero point is ``transverse", whose meaning will become clear later. We perturb the zero point \eqref{260602e7_88}, and consider a new orthogonal frame given by 
\begin{equation}\label{260602e7_90}
    \begin{split}
        & X= e_1+ \varepsilon_1 e_2+ \varepsilon_2 e_3+ O(\varepsilon^2), \\
        & Y= e_2- \varepsilon_1 e_1+ \varepsilon_3 e_3+ O(\varepsilon^2)\\
        & Z= e_3- \varepsilon_2 e_1- \varepsilon_3 e_2+ O(\varepsilon^2),
    \end{split}
\end{equation}
where 
\begin{equation}
    \varepsilon^2:= \varepsilon_1^2+ \varepsilon_2^2+\varepsilon_3^2.
\end{equation}
Define the map 
\begin{equation}
    \mathfrak{F}(\varepsilon_1, \varepsilon_2, \varepsilon_3)= \pnorm{
    A(X, Y), B(Z, X, Y), C(Z, Z, X, Y)
    },
\end{equation}
where $X, Y, Z$ are given by \eqref{260602e7_90}. One can compute directly that 
\begin{equation}
    \mathfrak{F}(\varepsilon_1, \varepsilon_2, \varepsilon_3)=
    \pnorm{
    -\varepsilon_1+ O(\varepsilon^2), 
    \varepsilon_2+ O(\varepsilon^2),
    2\varepsilon_1- \varepsilon_3+ O(\varepsilon^2)
    }.
\end{equation}
Note that 
\begin{equation}
    D \mathfrak{F}(0, 0, 0)=
    \begin{bmatrix}
        -1 & 0 & 0\\
        0 & 1 & 0\\
        2 & 0 & -1
    \end{bmatrix}
\end{equation}
and the determinant of $D\mathfrak{F}$ does not vanish. This is what we meant by that the zero point \eqref{260602e7_88} is transverse. The rest of the proof is immediate: If we consider a small smooth perturbation of the metric \eqref{260602e7_87}, then the new metric fails the chaotic curvature condition of order $\le 2$ by the implicit function theorem. This finishes the proof of the second part of Item (2) in Theorem \ref{existence}.

\subsection{Proof of the third item}

The proof is local, so we work in
a neighborhood of the origin in \(\mathbb R^3\). The main point is that the
failure of chaotic curvature of order \(\leq 3\) at a fixed point is detected by
the \(5\)-jet of the metric at that point. Recall first why \(5\)-jets are enough. The Ricci tensor contains two derivatives
of the metric. Therefore \(\nabla^m\operatorname{Ric}\) at a point is determined
by derivatives of the metric up to order \(m+2\). In particular,
\begin{equation}
\operatorname{Ric},\ \ 
\nabla\operatorname{Ric},\ \ 
\nabla^2\operatorname{Ric},\ \ 
\nabla^3\operatorname{Ric}
\end{equation}
at the origin are determined by the \(5\)-jet of the metric at the origin.\\

Let \(G^5_0\) denote the space of \(5\)-jets at the origin of Riemannian metrics
on \(\mathbb R^3\). Thus \(G^5_0\) is an open subset of a finite-dimensional
vector space. For \(\mathfrak g\in G^5_0\), and for a \(\mathfrak g(0)\)-orthonormal
frame \((X,E,Z)\) of \(T_0\mathbb R^3\), define
\begin{equation}
\mathcal R_m(\mathfrak g,X,E,Z)
:=
(\nabla^m\operatorname{Ric})_{\mathfrak g,0}
(\underbrace{Z,\ldots,Z}_{m},X,E),
\qquad 0\leq m\leq 3.
\end{equation}
For \(m=0\), this means
\begin{equation}
\mathcal R_0(\mathfrak g,X,E,Z)
=
\operatorname{Ric}_{\mathfrak g,0}(X,E).
\end{equation}
Define the bad set \(\Sigma_0\subset G^5_0\) by
\begin{equation}
\begin{aligned}
\Sigma_0
:=
\bigl\{
\mathfrak g\in G^5_0:\,
&\text{there exists a \(\mathfrak g(0)\)-orthonormal frame }
(X,E,Z)  \\
&\text{such that }
\mathcal R_m(\mathfrak g,X,E,Z)=0,\quad 0\leq m\leq 3
\bigr\}.
\end{aligned}
\end{equation}
Thus \(\mathfrak g\notin\Sigma_0\) means that, for every orthonormal frame
\((X,E,Z)\), at least one of
\begin{equation}
\operatorname{Ric}(X,E),\ \ 
(\nabla_Z\operatorname{Ric})(X,E),\ \ 
(\nabla_Z^2\operatorname{Ric})(X,E),\ \ 
(\nabla_Z^3\operatorname{Ric})(X,E)
\end{equation}
is nonzero at the origin.

We claim that \(G^5_0\setminus\Sigma_0\) is open and dense in \(G^5_0\). The
openness is straightforward by continuity and compactness. It remains to prove density. Consider the frame bundle
\begin{equation}
\mathcal B
:=
\bigl\{
(\mathfrak g,X,E,Z):\mathfrak g\in G^5_0,\,
(X,E,Z)\text{ is \(\mathfrak g(0)\)-orthonormal}
\bigr\}.
\end{equation}
If \(N=\dim G^5_0\), then
\begin{equation}
\dim \mathcal B=N+3,
\end{equation}
because the space of orthonormal frames in dimension three has dimension \(3\).
Define
\begin{equation}
\mathcal R:\mathcal B\to\mathbb R^4,
\qquad
\mathcal R:=(\mathcal R_0,\mathcal R_1,\mathcal R_2,\mathcal R_3).
\end{equation}
It is not difficult to see that  \(\mathcal R\) is a submersion, as the entries of $\mathcal R$ involve derivatives of the metric of different degrees. So, by the submersion theorem, or equivalently the implicit function theorem, we obtain that
\begin{equation}
\mathcal Z:=\mathcal R^{-1}(0)
\end{equation}
is a smooth submanifold of \(\mathcal B\) of codimension \(4\). In other words,
\begin{equation}
\dim\mathcal Z=(N+3)-4=N-1.
\end{equation}
The bad set \(\Sigma_0\) is the projection of \(\mathcal Z\) to \(G^5_0\). Since
all the defining relations are algebraic in the finite jet coefficients and in the
frame variables, \(\Sigma_0\) is semialgebraic. Therefore
\begin{equation}
\dim\Sigma_0\leq \dim\mathcal Z=N-1.
\end{equation}
In particular, \(\Sigma_0\) has empty interior in \(G^5_0\). This proves that
\(G^5_0\setminus\Sigma_0\) is dense. Once this is proven, we can use continuity and compactness again to finish the proof of item (3). \\

\noindent {\bf Acknowledgments.} S. G. is partly supported by the Nankai Zhide Foundation, NSFC Grant No. 12426204, and the New Cornerstone Science Foundation.  \\
Part of this work was carried out during Gong's visit to Nankai University; he acknowledges their hospitality with gratitude. S. Dai would like to thank Ruixiang Zhang for the discussion on the chaotic curvature condition.

Part of this work, especially Subsection 7.2, was developed with the assistance of large language models, including Doubao-Seed-2.0-pro, ChatGPT Pro 5.5, and Gemini Ultra 3.1-pro, at different stages of the project. In particular, Doubao-Seed suggested the first correct example for a model problem related to the second item of Theorem 2.7. Subsequently, ChatGPT was used to search for complete examples of the same type. All examples and arguments appearing in Subsection 7.2 were independently checked by the authors, who take full responsibility for their correctness.

\vspace{2cm}

\noindent S.~Dai,
\textsc{Center for Applied Mathematics and KL-AAGDM, Tianjin University, Tianjin, 300072, China}
\par\nopagebreak
  \textit{E-mail address}: \texttt{song.dai@tju.edu.cn}

\noindent L.~Gong, \textsc{Department of Mathematics, The Chinese University of Hong Kong, Hong Kong, China}
  \par\nopagebreak
  \textit{E-mail address}: \texttt{lwgong@math.cuhk.edu.hk}

\noindent S.~Guo, \textsc{Chern Institute of Mathematics, LPMC and New Cornerstone Science Laboratory, Nankai University, Tianjin, 300071, China}
\par\nopagebreak
  \textit{E-mail address}: \texttt{shaomingguo2018@gmail.com}

\end{document}